\documentclass[11pt,leqno,twoside]{article}

\usepackage[utf8]{inputenc} %
\usepackage[T1]{fontenc}    %
\usepackage[english]{babel} %
\usepackage{palatino}       %
\usepackage[a4paper,margin=.9in,marginpar=.9in]{geometry} %
\usepackage{xargs}          %
\usepackage{xparse}         %
\usepackage{csquotes}       %

\usepackage[shortlabels]{enumitem}

\usepackage[nottoc,numbib]{tocbibind} %

\usepackage{amssymb, amsmath, amsthm, amsfonts}
\usepackage{esint}
\usepackage{mathtools}

\usepackage[backend=biber,giveninits=true,minbibnames=5,maxbibnames=5]{biblatex}
\usepackage{xcolor}
\usepackage[pdfpagelabels]{hyperref}
\definecolor{green}{HTML}{2ECC71}
\definecolor{blue}{HTML}{3498DB}
\definecolor{red}{HTML}{E74C3C}
\hypersetup{%
colorlinks,
linktocpage,
urlcolor  = blue,
citecolor = green,
linkcolor = red}
\usepackage{listings}   %
\definecolor{lstgreen}{rgb}{0,0.5,0}
\usepackage[capitalise,nameinlink,sort]{cleveref} %
\AfterEndEnvironment{proof}{\noindent\ignorespaces} %
\makeatletter
\def\@endtheorem{\endtrivlist}
\makeatother
\Crefname{equation}{}{}
\Crefname{conditioni}{Condition}{Conditions}
\Crefname{conditionalti}{Condition}{Conditions}
\newtheorem{theorem}{Theorem}[section]
\Crefname{theorem}{Theorem}{Theorems}
\newtheorem{lemma}[theorem]{Lemma}
\Crefname{lemma}{Lemma}{Lemmas}
\newtheorem{proposition}[theorem]{Proposition}
\Crefname{proposition}{Proposition}{Propositions}
\newtheorem{corollary}[theorem]{Corollary}
\Crefname{corollary}{Corollary}{Corollaries}

\Crefname{conjecture}{Conjecture}{Conjectures}

\Crefname{assumption}{Assumption}{Assumptions}
\theoremstyle{definition}

\Crefname{definition}{Definition}{Definitions}

\Crefname{question}{Question}{Questions}
\theoremstyle{remark}
\newtheorem{remark}{Remark}
\Crefname{remark}{Remark}{Remarks}

\Crefname{example}{Example}{Examples}
\newtheorem*{example*}{Example}

\numberwithin{equation}{section}

\usepackage{calrsfs}
\DeclareMathAlphabet{\pazocal}{OMS}{zplm}{m}{n}

\newcommand{\CB}{\mathcal{B}}
\newcommand{\CC}{\mathcal{C}}

\newcommand{\CH}{\mathcal{H}}

\newcommand{\CL}{\mathcal{L}}
\newcommand{\CM}{\mathcal{M}}

\newcommand{\CP}{\mathcal{P}}

\newcommand{\cH}{\pazocal{H}}

\newcommand{\cM}{\pazocal{M}}

\newcommand{\cT}{\pazocal{T}}

\newcommand{\cW}{\pazocal{W}}

\newcommand{\lT}{\mathbf{T}}

\newcommand{\1}{\mathbf{1}}

\newcommand{\gB}{\mathfrak{B}}

\newcommand{\gZ}{\mathfrak{Z}}

\newcommand{\dE}{\mathbb{E}}

\newcommand{\dM}{\mathbb{M}}
\newcommand{\dN}{\mathbb{N}}

\newcommand{\dP}{\mathbb{P}}

\newcommand{\dR}{\mathbb{R}}

\newcommand{\dT}{\mathbb{T}}

\newcommand{\dW}{\mathbb{W}}

\newcommand{\dd}{\mathrm{d}}
\DeclareMathOperator{\e}{e}  %
\DeclareMathOperator{\vol}{vol} %
\DeclareMathOperator{\supp}{supp} %
\DeclareMathOperator{\Ent}{\dE\mathrm{nt}}

\DeclareDocumentCommand\tto{O{n} O{\infty} m}{\xrightarrow[{#1}\to{#2}]{#3}}

\DeclareSourcemap{
  \maps[datatype=bibtex]{
    \map{
      \step[fieldsource=doi,final]
      \step[fieldset=url,null]
    }  
  }
}

\author{Nathaël Gozlan, Ronan Herry, Giovanni Peccati}
\AtEndDocument{\bigskip\footnotesize%
  N.G.\ is a member of \textsc{MAP5 (UMR CNRS 8145), Université Paris Descartes}
  \par
  45 rue des Saints-Pères, 75270 Paris cedex 6, France.
  \par
  \href{mailto:natael.gozlan@parisdescartes.fr}{\texttt{natael.gozlan@parisdescartes.fr}} \par
  R.H.\ is a member of \textsc{Institut für Angewandte Mathematik, Bonn Universität}
  \par
  Endenicher Allee 60, D-53115, Bonn, Germany.
  \par
\href{mailto:herry@iam.uni-bonn.de}{\texttt{herry@iam.uni-bonn.de}}
\par
  G.P.\ is a member of \textsc{MRU, Université du Luxembourg}
  \par
  6 avenue de la Fonte, L-4363 Esch-sur-Alzette, Luxembourg.
  \par
\href{mailto:giovanni.peccati@uni.lu}{\texttt{giovanni.peccati@uni.lu}}%
}

\begin{document}
\title{Transport inequalities for random point measures}

\maketitle

\begin{abstract} 
  We derive transport-entropy inequalities for mixed binomial point processes, and for Poisson point processes.
  We show that when the finite intensity measure satisfies a Talagrand transport inequality, the law of the point process also satisfies a Talagrand type transport inequality.
  We also show that a Poisson point process (with arbitrary $\sigma$-finite intensity measure) always satisfies a universal transport-entropy inequality \`a la Marton.
  We explore the consequences of these inequalities in terms of concentration of measure and modified logarithmic Sobolev inequalities.
  In particular, our results allow one to extend a deviation inequality by Reitzner \cite{ReitznerConcentration}, originally proved for Poisson random measures with finite mass.
\end{abstract}
{\bf Keywords:} Point Processes, Transport-Entropy Inequalities, Concentration of Measure.\\
{\bf MSC2010:} 60G55; 60E15; 46E27. 

\section{Introduction}\label{section:intoduction}

In this paper, we establish new transport-entropy inequalities for various random point measures including the important case of Poisson random measures.

The investigation of transport-entropy inequalities starts in the nineties with works by Marton \cite{Marton96a, Marton96b} and by Talagrand \cite{Talagrand96}, in connection with the concentration of measure phenomenon for product measures.
We refer to \cite[Chapter 6]{Ledoux01}, \cite[Chapter 22]{Villani09}, \cite[Chapter 9]{BakryGentilLedoux}, and \cite{GozlanLeonard,GozlanSurvey} for general introductions and surveys on these intimately related topics.
In this field, the so-called \emph{Talagrand's transport inequality} and  \emph{Marton's universal transport inequality}, that we now briefly present, are of prime importance.

In order to state these inequalities, we introduce some notations.
Suppose that $(Z,d)$ is some complete separable metric space, and denote by $\CP(Z)$ the set of Borel probability measures on $Z$.
For $k\geq1$, the \emph{Monge-Kantorovich distance} $\cW_k$ (also called Wasserstein distance) between $\nu_1,\nu_2 \in \CP(Z)$ is defined by
\begin{equation}\label{eq:Wasserstein}
\cW_k^k(\nu_1,\nu_2) = \inf \dE \left[d^k(X_1,X_2)\right],
\end{equation}
where the infimum runs over the  set of all pairs $(X_1,X_2)$ of random variables such that $X_1 \sim \nu_1$ and $X_2 \sim \nu_2$.
The \emph{relative entropy} $\cH$ of $\nu_1$ with respect to $\nu_2$ is defined by 
\begin{equation}\label{eq:entropy}
\cH(\nu_1|\nu_2) = \int \log \frac{d\nu_1}{d\nu_2} \,\dd \nu_1,
\end{equation}
if $\nu_1 \ll \nu_2$, and $+\infty$ otherwise.

We say that a probability measure $\mu \in \CP(Z)$ satisfies \emph{Talagrand's transport inequality}\footnote{Here, to be precise, we refer to a symmetric version of Talagrand's inequality, involving two probability measures $\nu_1,\nu_2$.} with constant $a>0$ if, for all $\nu_1,\nu_2 \in \CP(Z)$, it holds
\begin{equation}\label{eq:Talagrand}
  \cW_2^2(\nu_1,\nu_2) \leq a \left[ \cH(\nu_1|\mu)+ \cH(\nu_2|\mu) \right].
\end{equation}
This inequality, first proved for the standard Gaussian measures in \cite{Talagrand96}, now plays a central role in the literature devoted to concentration of measure, coercive inequalities for Markov semigroups (logarithmic Sobolev and Poincaré inequalities, and their offshoots) \cite{OttoVillani,BobkovGentilLedoux}, and curvature-dimension conditions for weighted Riemannian manifolds or, more generallly, metric measured spaces \cite{Sturm1,Sturm2,LottVillani}.
Various sufficient conditions are available to ensure that a given probability measure satisfies \cref{eq:Talagrand} (see \cite{GozlanLeonard} for an overview and \cite{Gozlan12} for a necessary and sufficient condition when $Z=\dR$).
Dimension free concentration of measure is the main application of Talagrand's inequality (and of its variants): if a probability measure $\mu$ satisfies \cref{eq:Talagrand}, then, for any $n\geq1$, and for any vector $(X_1,\ldots,X_n)$ of i.i.d random variables with common law $\mu$, it holds
\begin{equation}\label{eq:dev-ineq}
\dP(f(X_1,\ldots,X_n) \geq t) \leq \e^{-\frac{t^2}{2a}},\qquad \forall t \geq 0,
\end{equation}
for every function $f \colon Z^n\to \dR$ which is of mean $0$ with respect to $\mu^{\otimes n}$ and $1$-Lipschitz with respect to the $\ell^2$ product distance on $Z^n$, 
\begin{equation}\label{eq:d_2}
d_2(x,y) = \left(\sum_{i=1}^n d^2(x_i,y_i)\right)^{1/2},\qquad x,y \in Z^n.
\end{equation}
We refer to \textit{e.g} \cite{GozlanLeonard} for a presentation of the nice general argument due to Marton that enables to deduce \cref{eq:dev-ineq} from \cref{eq:Talagrand}.
Remarkably, all the deviation inequalities in \cref{eq:dev-ineq} hold simultaneously with the same constant $a$ for all dimensions $n\geq 1$ (a property which actually characterizes Talagrand's inequality \cite{Gozlan09}). This type of dimension free bounds plays an important role in analysis, probability, or statistics in high dimensions \cite{Talagrand95,Ledoux01}. 

Marton's inequality involves a variant of the Monge-Kantorovich costs, which we denote by $\cM$ in the sequel and we define as follows:
\begin{equation}\label{eq:Martoncost}
\cM^2 (\nu_1 | \nu_2) = \inf \dE \left[\dP(X_1\neq X_2| X_2)^2\right],
\end{equation} 
where the infimum runs over the set of all pairs $(X_1,X_2)$ of random variables such that $X_1 \sim \nu_1$ and $X_2 \sim \nu_2$. We refer to  \cref{section:reminders_transport} for the presentation of the unifying framework of generalized transport costs introduced in \cite{GRSTGeneral} which contains in particular Monge-Kantorovich as well as Marton transport costs.
Marton's transport cost also admits the following explicit expression: if $\nu_1$ and $\nu_2$ are absolutely continuous with respect to some measure $\mu$ on $Z$, with $\nu_1 =f_1\mu$ and $\nu_2 = f_2\mu$, then the results of \cite{Marton96a} imply that
\begin{equation}\label{eq:Marton-explicit}
\cM^2 (\nu_1 | \nu_2) = \int \left[1 - \frac{f_1}{f_2}\right]_+^2 f_2\,\dd\mu.
\end{equation}
Contrary to Talagrand's inequality, which holds only for some specific probability measures, according to \cite{Marton96a}, all probability measures satisfy Marton's inequality.
A classical version of Marton's universal transport inequality reads as follows: for any probability measure $\mu$ on $Z$, it holds that
\begin{equation}\label{eq:Marton}
\cM^2 (\nu_1 | \nu_2) \leq 4 \cH(\nu_1|\mu)+ 4 \cH(\nu_2|\mu),
\end{equation}
for all $\nu_1,\nu_2 \in \CP(Z)$.
One can understand this inequality as a reinforcement of the classical Csiszar-Kullback-Pinsker inequality (see \textit{e.g} \cite{GozlanLeonard} and the references therein) comparing the squared total variation distance to relative entropy.
We refer to \cite{Dembo, Samson03, Samson07} for subsequent refinements of Marton's inequality.
Similarly to \cref{eq:Talagrand}, Marton's inequality has interesting consequences in terms of concentration of measure.
As Marton \cite{Marton96a} shows, \cref{eq:Marton} gives back the universal concentration of measure inequalities for product measures involving the so-called \enquote{convex distance} discovered by Talagrand in \cite{Talagrand95}.
To avoid entering into too technical details in this introduction, let us recall a more concrete consequence of \cref{eq:Marton} in terms of deviation inequalities for convex functions (see \cite{Marton96a,Dembo} for details).
Namely, if we equip $Z = \dR^p$ with the standard Euclidean norm and $\mu \in \CP(\dR^p) $ has a bounded support with diameter $D$, then for any $n\geq 1$, and for any vector $(X_1,\ldots,X_n)$ of i.i.d random variables with common law $\mu$, we have that
\begin{equation}\label{eq:dev-ineq2}
\dP(f(X_1,\ldots,X_n) \geq t) \leq e^{-t^2/4D^2},\qquad \forall t \geq 0,
\end{equation}
for all \emph{convex} or \emph{concave} function $f \colon (\dR^p)^n\to \dR$ which is of mean $0$ with respect to $\mu^{\otimes n}$ and $1$-Lipschitz with respect to the Euclidean norm on $(\dR^p)^n$.

Transport-entropy inequalities share a pivotal tensorization mechanism; this mechanism enables one to recover the dimension-free deviation bounds \cref{eq:dev-ineq} or \cref{eq:dev-ineq2} from \cref{eq:Talagrand} or \cref{eq:Marton}.
For instance, if $\mu$ satisfies Talagrand's inequality on a space $(Z,d)$ with a constant $a$, then its tensor product $\mu^{\otimes n}$ also satisfies Talagrand's inequality on the space $(Z^n,d_2)$ with $d_2$ given by \cref{eq:d_2} and with the same constant $a$.
This tensorization property, which comes from classical disintegration formulas for the functionals $\cW_2^2$ and $\cH$, is explained in full generality, for instance in \cite{GozlanLeonard}.
In this work, we use this stability by tensorization as a pervasive tool to obtain transport type inequalities for random point measures.

Let us now give a flavor of the content of this paper.
In brief, understanding what happens to Talagrand's and Marton's transport inequalities in the framework of random point processes serves as the basic motivation behind this work.
More precisely, in this article, we consider a particular class of point processes called \emph{mixed binomial point processes}; they are of the following form:
\begin{equation}\label{eq:mixed-binom}
\eta = \sum_{i=1}^{N} \delta_{X_i},
\end{equation}
where $N$ is a random variable taking values in $\dN := \{0,1,2,\ldots\}$, $(X_i)_{i\geq1}$ is a sequence of i.i.d random variables with values in $Z$ independent of $N$, and $\delta_{u}$ is the Dirac mass at $u \in Z$.
In the definition above, and everywhere in the rest of this work, by convention, an empty sum is equal to zero.
This random point process $\eta$ takes values in the set of all Borel finite measures on $Z$, written $\CM_b(Z)$ throughout the paper.
We usually denote by $\kappa \in \CP(\dN)$ the law of $N$, and by $\mu\in \CP(Z)$ the common law of the $X_i$'s, which we often refer to as the \enquote{sampling measure} of the process $\eta$.
We denote the law of $\eta$ by $B_{\mu,\kappa} \in \CP(\CM_b(Z))$.
Notably, when $N$ is a Poisson random variable with mean $\lambda\geq0$, the point process $\eta$ is a \emph{Poisson point process}, with (finite) intensity measure $\dE \eta = \nu = \lambda \mu \in \CM_b(Z)$.
We also denote by $\Pi_\nu \in\CP(\CM_b(Z))$ the law of such a Poisson point process.
We can also consider Poisson point processes with a $\sigma$-finite intensity measure; for the sake of simplicity, in this introduction, we only consider the finite intensity case.
We refer to \cref{section:reminders_point_processes} for further information and definitions on these elementary random point processes.

In this work, we highlight a general principle leading to transport-entropy inequalities for point processes: $B_{\mu,\kappa}$ inherits the transport inequalities satisfied by its sampling probability measure $\mu$.
To state a first representative result illustrating this general rule, let us introduce some notations.
Given $\nu_1,\nu_2 \in \CM_b(Z)$, we extend the definition given in \cref{eq:Wasserstein}:
\begin{equation}\label{eq:Wasserstein_measures}
\cW_2^2(\nu_1,\nu_2) = 
\begin{cases}
m \cW_2^2\left(\frac{\nu_1}{m},\frac{\nu_2}{m}\right), & \text{if}\ \nu_1(Z) = \nu_2(Z) = m>0; \\ 
0, & \text{if}\ \nu_1(Z) = \nu_2(Z) =0; \\
+\infty,& \text{if}\ \nu_1(Z) \neq \nu_2(Z).
\end{cases}
\end{equation}
Then, we define a process level Monge-Kantorovich cost $\dW_2^2$ on $\CP(\CM_b(Z))$ as follows: for any $\Pi_1,\Pi_2 \in \CP(\CM_b(Z))$, 
\begin{equation*}
\dW_2^2(\Pi_1,\Pi_2) = \inf \dE\left[\cW_2^2(\eta_1,\eta_2)\right],
\end{equation*}
where the infimum runs over the set of couples $(\eta_1,\eta_2)$ of random measures such that $\eta_1 \sim \Pi_1$ and $\eta_2 \sim \Pi_2$. 
The condition $\dW_2^2(\Pi_1,\Pi_2)<+\infty$ is a strong assumption. Indeed, if it holds then there exists a couple $(\eta_1,\eta_2)$, such that $\eta_1 \sim \Pi_1$ and $\eta_2 \sim \Pi_2$, and such that $\eta_1(Z) = \eta_2(Z)$ almost surely. 

With these notions at hand, we can now state an analogue of Talagrand's inequality for mixed binomial processes.
\begin{theorem}\label{thm:Talagrand}
Suppose that $\mu \in \CP(Z)$ satisfies Talagrand's inequality \cref{eq:Talagrand} with a constant $a>0$, then for any $\kappa \in \CP(\dN)$, the probability measure $B_{\mu,\kappa}\in \CP(\CM_b(Z))$ satisfies the following inequality: for all $\Pi_1,\Pi_2 \in \CP(\CM_b(Z))$ such that $\dW_2^2(\Pi_1,\Pi_2)<+\infty$,
\begin{equation}\label{equation:Talagrand_Poisson}
\dW_2^2(\Pi_1,\Pi_2) + 2a\cH(\lambda|\kappa) \leq a \cH(\Pi_1 | B_{\mu,\kappa}) + a \cH(\Pi_2 | B_{\mu,\kappa}),
\end{equation}
where $\lambda \in \CP(\dR_+)$ is such that for all $i\in \{1,2\}$, $\Pi_i( \{\eta \in \CM(Z) : \eta(Z) \in A\}) = \lambda(A)$, for all Borel $A \subset \dR_+$.
\end{theorem}
This result is another instance of a phenomenon highlighted in a paper by Erbar \& Huesmann \cite{ErbarHuesmann}: geometric functional inequalities lift from the base space $Z$ to the space of configurations.
More precisely, if $Z$ is a Riemannian manifold with Ricci curvature bounded by $K \in \dR$, with infinite volume measure $\vol$, and Riemannian distance $d$, the results of \cite{ErbarHuesmann} show that 
\begin{equation}\label{equation:ErbarHuesmann}
 \cH(\Pi_{t} | \Pi_{\vol}) \leq (1-t) \cH(\Pi_{0} | \Pi_{\vol}) + t \cH(\Pi_{1} | \Pi_{\vol}) - \frac{K}{2} t (1-t) \dW_2^2(\Pi_{0}, \Pi_{1}),
\end{equation}
where $\Pi_{\vol}$ denotes the Poisson point process with intensity measure $\vol$, and where $\Pi_{t}$ is any $\dW_{2}$-geodesic from $\Pi_{0}$ to $\Pi_{1}$, for $\Pi_{0}$ and $\Pi_{1} \in \CP(\CM_{\bar{\dN}}(Z))$ with $\CM_{\bar{\dN}}(Z)$ denoting the space of measures $\eta$ such that $\eta(K) \in \dN$ for all compact $K \subset Z$.
This means that the curvature property of the metric measured space $(Z,d,\vol)$ transfers to the (extended) metric measured space $(\CM_{\bar{\dN}}(Z), \dW_{2}, \Pi_{\vol})$.
Indeed, \cref{equation:ErbarHuesmann} states that $(\CM_{\bar{\dN}}(Z), \dW_{2}, \Pi_{\vol})$ has a synthetic Ricci curvature in the sense of Lott-Sturm-Villani bounded below by $K$ (see, \emph{e.g.} one of the seminal papers \cite{Sturm1,LottVillani} for definitions of synthetic Ricci lower bound in terms of convexity of the relative entropy, as well as \cite[Definition 9.1]{AGSMMS} for a definition in the extended metric measured spaces setting).
This lifting of Ricci lower bound is very natural, if we think of $\Pi_{\vol}$ as the invariant measure of a system of non-interacting $N$ Brownian motions on $Z$, where $N$ has a Poisson law with mean $\vol(M)$ (possibly infinite).
In the case where $K > 0$, the manifold is compact so that $\Pi_{\vol}$ has a representation as a mixed binomial process, and, provided the results of \cite{ErbarHuesmann} carry to the compact case, taking $t = \frac{1}{2}$, and using that the entropy is non-negative in \cref{equation:ErbarHuesmann} immediately yields, for all $\Pi_{0},\Pi_{1} \in \CP(\CM_{\bar{\dN}}(M))$ such that $\dW^{2}_{2}(\Pi_{0}, \Pi_{1}) < \infty$:
\begin{equation*}
  \frac{K}{4} \dW_{2}^{2}(\Pi_{0}, \Pi_{1}) \leq \cH(\Pi_{0} | \Pi_{\vol}) + \cH(\Pi_{1} | \Pi_{\vol}).
\end{equation*}
Let us stress that our argument works under the sole assumption that the space $Z$ supports a Talagrand inequality and is quite elementary.

To state an analogue of Marton's inequality, we introduce the following cost: for any $\Pi_1,\Pi_2 \in \CP(\CM_b(Z))$,
\begin{equation*}
  \dM^2(\Pi_1 | \Pi_2) = \inf \dE \left[  \int  \dE\left[ \left.\left[1- \frac{\eta_1(x)}{\eta_2(x)}\right]_+     \right | \eta_2\right]^2\,\eta_2(\dd x)              \right],
\end{equation*}
where $\eta_i(x)$ is a slight abuse of notation for $\eta_i(\{x\})$, and where the infimum runs over the set of couples $(\eta_1,\eta_2)$ of random measures such that $\eta_1 \sim \Pi_1$ and $\eta_2 \sim \Pi_2$.
\begin{theorem}\label{thm:Marton}
For any $\nu \in \CM_b(Z)$, the Poisson point process $\Pi_\nu$ satisfies the following inequality: for all $\Pi_1,\Pi_2 \in \CP(\CM_b(Z))$,
\begin{equation*}
\dM^2(\Pi_1 | \Pi_2) \leq 4 \cH(\Pi_1 |\Pi_\nu)+ 4 \cH(\Pi_2 |\Pi_\nu).
\end{equation*}
\end{theorem}
A similar statement actually holds for Poisson point processes associated with a $\sigma$-finite sampling measure $\nu$, see \cref{theorem:transport_entropy_poisson} for a precise result.

Let us mention that other authors derived universal transport-entropy inequalities for Poisson point processes.
For instance, Ma et.\ al \cite{MaShenWangWu} combine arguments of exponential integrability of \cite{BobkovGoetze} with the modified logarithmic Sobolev inequality for Poisson point processes of \cite{WuLSI} in order to derive some $L_1$-transportation inequalities for Poisson point processes.
To state their results, let us define
\begin{equation*}
  \dT_{TV}(\Pi_{1}, \Pi_{2}) = \inf \dE {\left\lVert \eta_{1} - \eta_{2} \right\rVert}_{TV},
\end{equation*}
where the infimum is over all $\eta_{1} \sim \Pi_{1}$ and all $\eta_{2} \sim \Pi_{2}$, and ${\lVert \, \cdot \, \rVert}_{TV}$ is the total variation norm.
Assuming, for simplicity, $\nu(Z) = 1$ and choosing $\phi = 1$ in \cite[Theorem 2.6]{MaShenWangWu} yields
\begin{equation}\label{equation:MSWW}
  \alpha(\dT_{TV}(\Pi, \Pi_{\nu})) \leq \cH(\Pi | \Pi_{\nu}),
\end{equation}
where $\alpha(r) = (1+r) \log(1+r) - r$.
In order to compare our inequalities in greater details, we would need to compare $\alpha(\dT_{TV})$ 
and $\dM^{2}$.
However, this does not seem to be possible in general. One reason for that is that since the $\eta_{i}$'s are not probability measures, one cannot apply Jensen's inequality. %

As mentioned earlier, Marton introduces her inequality so one may derive concentration of measure with respect to the so-called Talagrand convex distance.
In the setting of Poisson point processes, \cite{ReitznerConcentration} uses the Talagrand convex distance to prove concentration of measure results for Poisson random measures.
In \cref{section:concentration_binomial}, we recover, in the spirit of Marton's work, the results of \cite{ReitznerConcentration} using \cref{thm:Marton}.
Building on the ideas of \cite{BoucheronLugosiMassart}, \cite{BachmannPeccati} considers a different approach towards concentration of measure for Poisson point processes.
They obtain various general conditions on a functional $F \colon \CM_{b}(Z) \to \dR$ under which the random variable $F(\eta)$ satisfies some deviation inequality.
Since the space $\CM_{b}(Z)$ does not come with a natural distance, a rather involved technical condition replaces the condition of being Lipschitz that appeared in \cref{eq:dev-ineq}.
However, \cite{BachmannPeccati,BachmannReitzner} shows that the so-called \emph{geometric $U$-statistics} always satisfy this condition, and hence always satisfy some concentration of measure estimate.
Based on \cref{thm:Marton}, we recover a deviation inequality for $U$-statistics in the spirit of \cite{BachmannPeccati,BachmannReitzner} with a simple argument.
It would be interesting to know whether we could recover the main result of \cite{BachmannPeccati} for generic functionals from our transport inequality.
This question will be examined elsewhere.

\cref{thm:Talagrand,thm:Marton} are consequences of more general results presented in \cref{sec:Talagrand,sec:Marton}.
As already mentioned above, the tensorization property of the transport inequality satisfied by the sampling measure on the base space is crucial in our analysis.
Heuristically, this dimension free structure matters, since a random point process $\eta$ as in \cref{eq:mixed-binom} can be seen as a random vector of random dimension (up to the permutation of its coordinates). 
It should be noted that in the work \cite{ErbarHuesmann} on Ricci curvature bounds on the configuration space, even if the arguments are different, tensorization properties also play a pivotal role (in that case the tensorization of the Bakry-Emery condition).

The paper is organized as follows.
\cref{section:preliminaries} contains generalities about spaces of measure, optimal transport and functional inequalities.
\cref{section:topology} recalls some basic facts about the weak topology of the space of measures over a Polish space.
\cref{section:reminders_transport,section:transport_entropy} present, in a self-contained way, the structural properties of the generalized optimal transport and the related transport-entropy inequalities introduced by \cite{GRSTGeneral}, and the subsequent results of \cite{BackhoffBeiglboeckPammer}.
This framework encompasses, as a particular case, Talagrand and Marton inequalities.
In \cref{section:ancillary}, we recall two further stability properties of transport-entropy inequalities: stability by tensorization (as already explained this property is at the heart of our argument), stability by push-forward, and stability by approximation.
\cref{section:transport_ineq_point_process} is the core of this article and contains our main results that are \cref{theorem:Talagrand_binomial,theorem:Marton_binomial,theorem:transport_entropy_poisson}.
We start, in \cref{section:reminders_point_processes}, with some reminders on binomial point processes and Poisson point processes.
In \cref{sec:Talagrand}, we prove \cref{theorem:Talagrand_binomial}, that states that a Talagrand inequality for $\mu$ implies a Talagrand-type inequality for $B_{\mu,\kappa}$, and that implies \cref{thm:Talagrand}.
We first show our result for the simple case of binomial processes of deterministic size, and then extend it to random size using some properties of optimal transport and relative entropy under conditioning.
\cref{sec:Marton} contains several results about Marton-type transport-entropy inequalities on the configuration space.
Our first result in that direction is \cref{theorem:Marton_binomial}, that states that a Marton inequality for $\mu$ implies a transport-entropy inequality for binomial process of deterministic size.
Since, for a properly chosen distance, all probability measures satisfy a Marton inequality, this produces \cref{theorem:transport_entropy_binomial} stating a universal Marton inequality for binomial process of fixed size.
Contrary to the case of Talagrand inequality, we cannot extend our results to mixed binomial processes of arbitrary random size.
In \cref{theorem:transport_entropy_poisson}, we extend \cref{theorem:transport_entropy_binomial} to Poisson point processes with $\sigma$-finite intensity measure using a strong approximation of the Poisson point process by thinning a binomial process.
This result implies \cref{thm:Marton}.
In \cref{section:concentration_binomial}, we study the consequences of our transport-entropy inequalities in terms of concentration of measure.
In \cref{sectopn:concentration_generic}, we explain how to recover the main findings of \cite{ReitznerConcentration} from \cref{theorem:transport_entropy_poisson}.
In \cref{section:concentration_convex}, we also discuss concentration of measure for a particular class of functionals, containing in particular geometric $U$-statistics, and extend some results of \cite{BachmannPeccati, BachmannReitzner} to this class of functionals. 
Finally, in \cref{section:log_sob} we discuss the consequences of \cref{theorem:Marton_binomial} in terms of modified logarithmic Sobolev inequality.
Following \cite{GRSTGeneral}, \cref{theorem:log_sob_Rc} is a general modified logarithmic Sobolev inequality on the Poisson space where the energy term is given in term of an infimum-convolution operator. 
\cref{theorem:log_sob_monotonic} shows how we can partially deduce from this general modified logarithmic Sobolev inequality the modified logarithmic Sobolev inequality of \cite{WuLSI} for monotonic functionals. 

\section*{Acknowledgments}
N.G.\ is supported by a grant of the Simone and Cino Del Duca Foundation;
R.H.\ gratefully acknowledges support of the European Union through the European Research Council Advanced Grant for \emph{Karl-Theodor Sturm} \enquote{Metric measure spaces and Ricci curvature – analytic, geometric and probabilistic challenges};
G.P.\ is supported by the FNR grant FoRGES (R-AGR-3376-10) at Luxembourg University.
\tableofcontents%
\section{Preliminaries on measures and optimal transport}\label{section:preliminaries}
In what follows, $(E,d)$ is a complete and separable metric space. We recall below basic topological properties of some sets of measures on $E$, then we introduce the framework and the basic properties of (generalized) optimal transport costs between two measures on $E$ and finally, we recall definitions and basic properties of transport-entropy inequalities.
In the subsequent sections, we apply the material of the present section with $E=Z$, $Z$ being the state space of our point process, or with $E = \CM_b(Z)$.

\subsection{Topology of spaces of measures}\label{section:topology}

We always regard the space $E$ as a measurable space equipped with its Borel $\sigma$-algebra $\gB(E)$.
We recall that we denote by $\CM_{b}(E)$ the space of finite measures and by $\CP(E)$ the set of probability measures on $E$.
We also write $\CM_{0}(E)$ for the space of Radon measures that are finite on balls.
We write $\CC_{b}(E)$ for the space of bounded continuous functions, and $\CC_{0}(E)$ for the space of continuous functions vanishing outside of a ball. 
We endow the sets $\CM_{b}(E)$ and $\CP(E)$ with the \emph{weak} topology that is generated by the maps $f^{*} \colon \nu \mapsto \int f\,\dd\nu$ with $f \in \CC_{b}(E)$.
On the other hand, we endow the space $\CM_{0}(E)$ with the \emph{vague} topology that is generated by the $f^{*}$ with $f \in \CC_{0}(E)$.
According to \cite[Section 4.1]{Kallenberg}, $\CM_{b}(E)$ and $\CP(E)$ (with the weak topology), and $\CM_{0}(E)$ (with the vague topology) are Polish spaces (complete, separable, and metrizable).
We often work with the following subset of $\CM_{b}(E)$:
\begin{equation}\label{eq:configuration}
\CM_{\dN}(E) = \left\{ \sum_{i=1}^k \delta_{a_i} : k \in \dN,\, a_1,\ldots,a_k \in E\right\}
\end{equation}
corresponding to finite configurations over $E$.
We also work with $\CM_{\bar{\dN}}(E)$ consisting of all $\xi \in \CM_{0}(E)$ such that, for every closed ball $B \subset E$, $\xi_{\restriction B} \in \CM_{\dN}(B)$.
The spaces $\CM_{\dN}(E)$ or $\CM_{\bar{\dN}}(E)$ are the natural state space of our point processes. 
\begin{lemma}\label{lem:configuration}
  The set $\CM_{\dN}(E)$, resp.\ $\CM_{\bar{\dN}}(E)$, is a closed subset of $\CM_b(E)$, resp.\ $\CM_{0}(E)$.
  In particular, $\CM_{\dN}(E)$ and $\CM_{\bar{\dN}}(E)$ are a Polish spaces.
\end{lemma}
\begin{proof}
  Let $\{\xi_n;\; n \in \dN\} \subset \CM_{\dN}(E)$ converging to some $\xi \in \CM_{b}(E)$.
  Let us show that $\xi$ belongs to $\CM_{\dN}(E)$.
  Since $\xi_n(E)$ is a sequence of integers converging to $\xi(E)$, it must be that $k:=\xi(E) \in \dN$ and that $\xi_n(E)=k$ for all $n$ sufficiently large.
  Then, since $(\xi_n)_{n\geq0}$ is converging, it is according to Prohorov's theorem a tight sequence: for any $\varepsilon>0$, there exists a compact set $K_{\varepsilon}\subset E$ such that $\sup_{n\geq0}\xi_n(E\setminus K_\varepsilon) \leq \varepsilon$.
  Taking in particular $\varepsilon =1/2$ and using the fact that the $\xi_n$'s are sums of Dirac masses, one sees that $\xi_n(E\setminus K_{1/2}) = 0$.
  In other words, for all $n$ large enough $\xi_n = \sum_{i=1}^k \delta_{x_i^n}$, with $x_1^n,\ldots,x_k^n \in K_{1/2}$.
  Using a diagonal extraction argument, one can assume without loss of generality that $x_i^n \to x_i \in K_{1/2}$ for all $i \in \{1,\ldots,k\}$.
  This easily implies that $\xi = \sum_{i=1}^k \delta_{x_i}$ and so $\xi \in \CM_{\dN}(E)$.
  Now let $\{\xi_{n};\; n \in \dN \} \subset \CM_{\bar{\dN}}(E)$ converging to some $\xi \in \CM_{0}(E)$.
  By what precedes for every ball $B \subset E$, there exists $\xi_{B} \in \CM_{\dN}(E)$ such that
  \begin{equation*}
    {\xi_{n}}_{\restriction B} \tto{weakly} \xi_{B}.
  \end{equation*}
  By a compatibility argument, we have that $\xi_{\restriction B} = \xi_{B}$; hence $\xi \in \CM_{\bar{\dN}}(E)$.
  As closed subsets of a Polish spaces, $\CM_{\dN}(E)$ and $\CM_{\bar{\dN}}(E)$ are themselves Polish.
\end{proof}

\subsection{Optimal transport between measures of arbitrary masses}\label{section:reminders_transport}
\subsubsection{Couplings, transport costs and generalized transport costs}
A \emph{coupling} of $\nu_{1}$ and $\nu_{2} \in \CM_{0}(E)$ is an element $N \in \CM_{0}(E \times E)$ such that, for all $A \in \gB(E)$, $N(A \times E) = \nu_{1}(A)$ and $N(E \times A) = \nu_{2}(A)$.
Observe that if $\nu_{1},\nu_{2} \in \CM_{b}(E)$, then $N \in \CM_{b}(E \times E)$.
Imposing $\nu_{1}(E) = \nu_{2}(E)$ is a necessary and sufficient condition for the existence of a coupling between $\nu_{1}$ and $\nu_{2}$.
Indeed, if there exists a coupling $N$ of $\nu_{1}$ and $\nu_{2}$ then $\nu_{1}(E) = N(E \times E) = \nu_{2}(E)$.
On the other hand, if $\nu_{1}$ and $\nu_{2} \in \CM_{0}(E)$ with $\nu_{1}(E) = \nu_{2}(E) = n \in \dN \cup \{\infty\}$, then we can write $\nu_{1} = \sum_{q = 1}^{n} \nu_{1,q}$ and $\nu_{2} = \sum_{q = 1}^{n} \nu_{2,q}$ with $\nu_{i,q} \in \CP(E)$.
Then $N = \sum_{q=1}^{n} (\nu_{1,q} \otimes \nu_{2,q})$ is a coupling of $\nu_{1}$ and $\nu_{2}$.
According to \cite[Proposition 13, Section 2.7]{Bourbaki}, for every coupling $N$, there exists a measurable application $E \ni x \mapsto p_{x} \in \CP(E)$ such that for all $A$ and $B \in \gB(E)$:
\begin{equation}
  N(A \times B) = \int_{A} p_{x}(B) \nu_{1}(\dd x).
\end{equation}
We call $p = \{p_{x};\; x \in E\}$ the \emph{disintegration kernel} of $N$ along $\nu_{1}$.
For short, we often abbreviate 
\begin{equation*}
N(\dd x \dd y) = p_{x}(\dd y) \nu_{1}(\dd x).
\end{equation*}

Following \cite{GRSTGeneral}, given a bi-measurable \emph{cost function} $c \colon E \times \CP(E) \to [0,\infty]$, the \emph{(generalized) optimal transport cost} associated to $c$ from $\nu_{1}$ to $\nu_{2} \in \CM_{0}(E)$, denoted by $\cT_{c}(\nu_{2} | \nu_{1})$, is given by
\begin{equation}\label{equation:transportation_cost}
    \cT_{c}(\nu_{2}|\nu_{1}) = \inf \int c(x,p_{x})\,\nu_{1}(\dd x),
\end{equation}
where the infimum runs over all couplings $N(\dd x \dd y) = p_{x}(\dd y) \nu_{1}(\dd x) $ between $\nu_{1}$ and $\nu_{2}$.
We use here the convention that $\inf \emptyset = \infty$. In particular, we have $\cT_{c}(\nu_{2} | \nu_{1}) = \infty$ as soon as $\nu_{1}(E) \neq \nu_{2}(E)$.
We always implicitly assume that the bi-measurability of $c$ is satisfied in the rest of the document.
Note that, $\cT_{c}$ is, in general, not symmetric with respect to $\nu_{1}$ and $\nu_{2}$.
When $c$ is of the form 
\begin{equation}\label{equation:omega_cost}
  c(x,p) = \int \omega(x, y) p(\dd y),
\end{equation}
for some $\omega \colon E \times E \to [0,\infty]$, we say, with a slight abuse of language, that $c$ is \emph{linear}, and we set $\cT_{\omega} = \cT_{c}$.
The cost $\cT_{\omega}$ is the transportation cost associated to $\omega$ in the usual sense of optimal transport:
\begin{equation}
  \cT_{\omega}(\nu_{2} | \nu_{1}) = \inf \int \omega(x,y) N(\dd x \dd y),
\end{equation}
where the infimum is running over all couplings $N$ of $\nu_{1}$ and $\nu_{2}$.
Such objects have been intensively studied and play an important role in many different areas of mathematics.
The reader can look at the reference \cite{Villani09} and the references therein.

\subsubsection{Some examples}
Let us recall some classical choices for transport costs.

\paragraph{Monge-Kantorovich distance}
Choosing $\omega = d^{k}$ with $k\geq 1$ yields $\cT_{d^{k}} = \cW^{k}_{k}$, the Monge-Kantorovich transport cost already introduced in \cref{eq:Wasserstein} and \cref{eq:Wasserstein_measures}.

\paragraph{Marton-type transport distances}
Given a measurable function $\rho \colon E \times E \to [0,\infty]$ and a convex function $\alpha \colon \dR^+ \to \dR^+$, we introduce the Marton type cost function 
\begin{equation}\label{equation:Marton_cost}
c(x,p) = \alpha\left(\int \rho(x,y) p(\dd y)\right), \quad x \in E,\, p \in \CP(E),
\end{equation}
and, following the notations of \cite{GRSTGeneral}, we denote by $\widetilde{\cT}_{\alpha,\rho}$ the associated optimal transport cost, namely
\begin{equation*}
  \cT_{\alpha, \rho}(\nu_{2} | \nu_{1}) = \inf \int \alpha\left(\int \rho(x,y) p_{x}(\dd y)\right) \nu_{1}(\dd x),
\end{equation*}
where the infimum runs over all couplings $N(\dd x \dd y) = p_{x}(\dd y) \nu_{1}(\dd x)$.
When $\nu_{1}$, $\nu_{2} \in \CP(E)$, we can also write
\begin{equation}\label{equation:Marton_transportation_cost}
\widetilde{\cT}_{\alpha,\rho}(\nu_{2} | \nu_{1}) = \inf\dE \left[\alpha\left(\dE\left[ \rho(X_{1},X_{2})  | X_{1} \right]\right)\right],
\end{equation}
where the infimum runs over the couples $(X_1,X_2)$ with $X_1\sim \nu_1$ and $X_2 \sim \nu_2$. In particular, taking the \emph{Hamming distance} $\rho(x,y) = \1_{x\neq y}$ and $\alpha(x) = x^2$, $x\geq0$, gives back Marton's cost $\cM^2$ already introduced in  \cref{eq:Martoncost} and  \cref{eq:Marton-explicit}.

The costs $\cW_{p}$ are symmetric since $d$ is symmetric.
More generally $\cT_{\omega}$ is symmetric provided $\omega$ is symmetric.
On the other hand $\widetilde{\cT}_{\alpha,\rho}$ is, in general, not symmetric even when $\rho$ is symmetric.

Applying Jensen's inequality in  \cref{equation:Marton_cost}, we easily check that, for all $\nu_{1},\nu_{2} \in \CM_{b}(E)$,
\begin{equation}\label{equation:Jensen_costs}
 \widetilde{\cT}_{\alpha,\rho}(\nu_{2} | \nu_{1}) \leq \cT_{\alpha\circ\rho}(\nu_{1}, \nu_{2});
\end{equation}
moreover, if $\nu_{1},\nu_{2} \in \CP(E)$ using Jensen's inequality in the integral with respect to to $\nu_{1}$ in the definition of $\widetilde{\cT}_{\alpha,\rho}$ shows that:
\begin{equation}\label{equation:Jensen_costs_proba}
\alpha\left(\cT_{\rho}(\nu_{2} | \nu_{1})\right)\leq \widetilde{\cT}_{\alpha,\rho}(\nu_{2} | \nu_{1}).
\end{equation}

Choosing the \emph{Hamming distance} $\rho(x,y)=d_{H}(x,y) = \1_{x \ne y}$, $x,y\in E$, we have for any $\nu_{1},\nu_{2} \in \CP(E)$:
\begin{equation*}
  \cT_{d_{H}}(\nu_{2} | \nu_{1}) = \cT_{d_{H}^2}(\nu_{2} | \nu_{1}) = \inf_{X_1 \sim \nu_1,X_2 \sim \nu_2} \dP(X_{1} \ne X_{2})= \|\nu_1-\nu_2\|_{TV},
\end{equation*}
where ${\|\cdot\|}_{TV}$ is the classical \emph{total variation} norm.
In this case, \cref{equation:Jensen_costs,equation:Jensen_costs_proba} yield 
\begin{equation*}
\|\nu_2-\nu_1\|_{TV} ^2\leq \cM^2(\nu_2|\nu_1) \leq \|\nu_2-\nu_1\|_{TV}.
\end{equation*}

\subsubsection{Existence of optimal couplings, lower semicontinuity of transport costs and stability of optimal couplings}
Backhoff-Veraguas et.\ al \cite{BackhoffBeiglboeckPammer} prove the following result (the authors work with $\nu_{1}$, $\nu_{2} \in \CP(E)$ but the generalization to $\CM_{b}(E)$ presents no difficulty).
\begin{theorem}\label{theorem:existence_weak_coupling}
  Let $c \colon E \times \CP(E) \to [0,\infty]$ be convex in its second argument, and jointly lower semi-continuous, then, for all $\nu_{1}$ and $\nu_{2} \in \CM_{b}(E)$ with same total mass, there exists a coupling $N(\dd x \dd y) = p_{x}(\dd y) \nu_{1}(\dd x)$ of $\nu_{1}$ and $\nu_{2}$ with disintegration kernel $p$ such that
  \begin{equation}
    \cT_{c}(\nu_{2} | \nu_{1}) = \int c(x, p_{x}) \nu_{1}(\dd x).
  \end{equation}
\end{theorem}
\begin{proof}Let $m=\nu_1(E)= \nu_2(E)$ be the common mass of the measures. If $m=0$, then $\nu_1=\nu_2=0$ and $\cT_c(\nu_2|\nu_1)=0$ and the coupling $N =0$ is optimal.
  Now, assume that $m>0$.
  Since $\cT_{c}(\nu_{2} | \nu_{1}) = m \cT_{c}(\nu_{2}/m | \nu_{1}/m)$, we use \cite[Theorem 1.1]{BackhoffBeiglboeckPammer} to conclude.
\end{proof}

Rather classically, this yields lower semi-continuity of $\cT_{c}$ as shown in \cite{BackhoffBeiglboeckPammer}.
Below we simply adapt the argument to finite measures of arbitrary total mass.
\begin{theorem}\label{theorem:lsc_transport_cost}
  Let $c \colon E \times \CP(E) \to [0,\infty]$ be jointly lower semi-continuous, and convex in its second argument, then $\cT_{c} \colon \CM_{b}(E) \times \CM_{b}(E) \to [0,\infty]$ is lower semi-continuous.
\end{theorem}
\begin{proof}
Let $(\mu_{k})_{k\geq0}$ and $(\nu_{k})_{k\geq0}$ respectively weakly converging to $\mu$ and $\nu$ in $\CM_{b}(E)$.
Let us show that 
\begin{equation}\label{eq:liminf}
\liminf_{k\to +\infty} \cT_c(\mu_k|\nu_k)  \geq \cT_c(\mu|\nu).
\end{equation}
By definition of the weak convergence, we have $\mu_k(E) \to \mu(E)$ and $\nu_k(E) \to \nu(E)$.
If $\mu(E) \neq \nu(E)$, then for all $k$ sufficiently large, it holds $\mu_k(E) \neq \nu_k(E)$ and so $\cT_c(\mu_k|\nu_k) = +\infty$.
Thus  \cref{eq:liminf} holds in this case.
If $\mu(E) = \nu(E)=0$, then $\cT_c(\mu|\nu)=0$ and  \cref{eq:liminf} trivially holds.
Let us now assume that $\mu(E) = \nu(E) =m>0$. Extracting a subsequence if necessary, one can assume that $\mu_k(E) = \nu_k(E) = m_k>0$ for all $k \geq 0$.
Since $\cT_{c}(\mu_{k} | \nu_{k}) = m_k \cT_{c}(\mu_{k}/m_k | \nu_{k}/m_k)$ and using \cite[Theorem 2.9]{BackhoffBeiglboeckPammer}, we see that
\begin{equation*}
\liminf_{k\to +\infty} \cT_c(\mu_k|\nu_k) =\liminf_{k\to +\infty} m_k \cT_{c}(\mu_{k}/m_k | \nu_{k}/m_k) \geq m \cT_{c}(\mu/m | \nu/m) =  \cT_{c}(\mu| \nu),
\end{equation*}
which completes the proof.
\end{proof}
When $c$ is lower semi-continuous, we do not know if $\cT_{c} \colon \CM_{0}(E) \times \CM_{0}(E) \to [0,\infty]$ is lower semi-continuous.

Finally, let us recall the following stability result taken from~\cite[Theorem 4.6]{Villani09}.
\begin{theorem}\label{thm:stability}
  Let  $\omega \colon E \times E \to [0,\infty]$ be a lower semi-continuous cost function.
  Let $\mu,\nu \in \CP(E)$ be such that  $\cT_{\omega}(\mu, \nu) < \infty$ and let $N$ be an optimal transport plan.
  Let $\tilde{N} \in \CM_{b}(E \times E)$ such that $\tilde{N} \leq N$ and $\tilde{N} \ne 0$, then the probability measure
  \begin{equation}
    N' = \frac{\tilde{N}}{\tilde{N}(E \times E)},
  \end{equation}
  is an optimal transport plan between its marginals.
\end{theorem}

\subsubsection{Optimal (partial) transport between configurations}\label{section:transport_configuration_space}
In this section, we study in more details $\cT_{c}(\chi | \xi)$, in the case where $\xi,\chi \in \CM_{\bar{\dN}}(E)$ are configurations.
\begin{proposition}\label{prop:birkhoff}
  Let $\omega \colon E \times E \to [0,\infty]$ be lower semi-continuous.
  Let $\xi,\chi \in \CM_{\dN}(E)$ such that $\xi = \sum_{i=1}^k \delta_{a_i}$ and $\chi = \sum_{i=1}^k \delta_{b_i}$ for some $k\geq 1$ and $a_1,\ldots,a_k$, $b_1,\ldots,b_k \in E$. 
  Then,
  \begin{equation}
    \cT_{\omega}(\chi | \xi) = \inf \left\{ \sum_{i=1}^{k} \omega(a_{i}, b_{\sigma(i)})\right\}.
  \end{equation}
  where the infimum runs over the set of all permutations $\sigma$ of $\{1,\ldots,k\}$.
\end{proposition}
This result is very classical (but usually stated for pairwise distinct $a_i$'s or $b_j$'s); we include a proof for completeness.
\begin{proof}
The inequality $\leq$ is clear, since for any permutation $\sigma$, the measure $N = \sum_{i=1}^n \delta_{(a_i,b_{\sigma(i)})}$ is a coupling between $\xi$ and $\chi$.
To prove the converse, let us consider a coupling $N$ between $\xi$ and $\chi$ and define the matrix $M= [M_{i,j}]_{1\leq i,j\leq k}$ as follows:
\begin{equation*}
M_{i,j} = \frac{N(\{a_i\} \times \{b_j\})}{\mu(a_i)\nu(b_j)},\qquad \forall 1\leq i,j \leq k.
\end{equation*}
The matrix $M$ is bi-stochastic, since denoting by $S= \{b_1,\ldots,b_k\}$ the support of $\chi$, 
\begin{align*}
\sum_{j=1}^kM_{i,j} &= \frac{1}{\xi(a_i)}\sum_{j=1}^k \frac{N(\{a_i\} \times \{b_j\})}{\chi(b_j)} = \frac{1}{\xi(a_i)} \int_{S} \frac{N(\{a_i\} \times \{y\})}{\chi(y)}\,\chi(\dd y)\\ &=  \frac{1}{\xi(a_i)} \sum_{y\in S} N(\{a_i\} \times \{y\}) = 1
\end{align*}
and similarly, $\sum_{i=1}^kM_{i,j}=1$.
Since $\int \omega(x,y)\,N(\dd x \dd y) = \sum_{i,j} \omega(x_i,y_i) M_{i,j}$, the other inequality follows from Birkhoff Theorem on extremal points of doubly stochastic matrices.
\end{proof}

Even when $c$ is bounded, the cost $\cT_{c}$ is infinite as soon as the measures are of different total masses. One way to avoid this, is to work with \emph{partial} optimal transport that we now define.
For $\xi$, $\chi \in \CM_{\dN}(E)$, we define
\begin{equation}\label{equation:partial_transport_cost}
  \cT_{c,0}(\chi | \xi) =
  \begin{cases}
    \min \left\{ \cT_{c}(\chi | \xi'): \xi'\in \CM_{\dN}(E),\, \xi' \leq \xi,\, \xi'(Z) = \chi(Z) \right\}, & \text{if}\ \xi(Z) \geq \chi(Z); \\
    \min \left\{ \cT_{c}(\chi' | \xi):  \chi'\in \CM_{\dN}(E), \chi' \leq \chi,\, \chi'(Z) = \xi(Z) \right\}, & \text{if}\ \xi(Z) < \chi(Z).
  \end{cases}
\end{equation}
Here and in the rest of this paper, the notation $\nu' \leq \nu$, for $\nu$ and $\nu' \in \CM_{0}(E)$, means that $\nu'(A) \leq \nu(A)$ for every Borel set $A$.
The following result follows immediately from  \cref{prop:birkhoff}.
\begin{proposition}
  Let $\xi,\chi \in \CM_{\dN}(E)$ be such that $\xi = \sum_{i=1}^n \delta_{a_i}$ and $\chi = \sum_{i=1}^k \delta_{b_i}$, for some $k,n\geq 1$ and $a_1,\ldots,a_n$, $b_1,\ldots,b_k \in E$.
  Assume $n \geq k$ then
  \begin{equation}
    \cT_{\omega,0}(\xi, \chi) = \inf \left\{ \sum_{i=1}^{k} \omega(x_{i}, y_{i}): \sum_{i=1}^{k} \delta_{x_{i}} \leq \xi,\, \chi = \sum_{i=1}^{k} \delta_{y_{i}} \right\}.
  \end{equation}
\end{proposition}

\subsection{Transport-entropy inequalities}\label{section:transport_entropy}
\subsubsection{Definitions and examples}\label{sec:defTI}
Recall the definition of the relative entropy functional $\cH$ given in  \cref{eq:entropy}.
Given a  cost function $c \colon E \times \CP(E) \to [0,\infty]$, we say that a probability measure $\gamma \in \CP(E)$ satisfies the \emph{transport-entropy inequality} $(\lT_{c})$ with constants $a_{1}, a_{2} > 0$ if
\begin{equation}\label{equation:transport_entropy-c}\tag{$\lT_{c}$}
  \cT_{c}(\nu_{2}|\nu_{1}) \leq a_{1} \cH(\nu_{1}|\gamma) + a_{2} \cH(\nu_{2}|\gamma), \qquad \forall \nu_{1}, \nu_{2} \in \CP(E).
\end{equation}
When $c(x, p) = \int \omega(x,y) p(\dd y)$, for $\omega \colon E \times E \to [0,\infty]$ we use ($\lT_{\omega}$) to designate the corresponding transport-entropy inequality, which reads as follows
\begin{equation}\label{equation:transport_entropy-omega}\tag{$\lT_{\omega}$}
  \cT_{\omega}(\nu_{1},\nu_{2}) \leq a_{1} \cH(\nu_{1}|\gamma) + a_{2} \cH(\nu_{2}|\gamma), \qquad \forall \nu_{1}, \nu_{2} \in \CP(E).
\end{equation}
When $c(x,p) = \alpha\left(\int \rho(x,y)p(\dd y)\right)$ for some cost function $\rho \colon E\times E \to [0,\infty]$ and some convex function $\alpha$ as in \cref{equation:Marton_cost} on $\dR^+$, we say that $\gamma$ satisfies the inequality $(\widetilde{\lT}_{\alpha,\rho})$ with constants $a_{1}, a_{2} > 0$ if
\begin{equation}\label{equation:Marton}\tag{$\widetilde{\lT}_{\alpha,\rho}$}
  \widetilde{\cT}_{\alpha,\rho}(\nu_{2}|\nu_{1}) \leq a_{1} \cH(\nu_{1}|\gamma) + a_{2} \cH(\nu_{2}|\gamma), \qquad \forall \nu_1,\nu_2 \in \CP(E).
\end{equation}
In these definitions, $a_{1}$ or $a_{2}$ can assume the value $\infty$, and we use the convention that $0 \cdot \infty = 0$.
For instance, if, for instance, $a_{2} = \infty$ the previous inequalities are non-trivial if and only if we take $\nu_{2} = \gamma$.
In that case, \cref{equation:Marton} is interpreted as
\begin{equation}\tag{$\widetilde{\lT}_{\alpha,\rho}$}
  \widetilde{\cT}_{\alpha,\rho}(\gamma|\nu_{1}) \leq a_{1} \cH(\nu_{1}|\gamma), \qquad \forall \nu_1 \in \CP(E).
\end{equation}
When $a_{1} = a_{2} = a < \infty$, we simply say that $\gamma$ satisfies a transport-entropy inequality with constant $a$.

We refer to \cite{GozlanLeonard,GRSTGeneral} for a panorama of cost functions and transport inequalities.
We refer to inequalities of the form \cref{equation:transport_entropy-omega} as \emph{Talagrand type} transport-entropy inequalities; whereas we use \emph{Marton type} transport-entropy inequalities for inequalities of the form \cref{equation:Marton}.

\paragraph{Talagrand's inequality}
Choosing $\omega = d^{2}$ in \cref{equation:transport_entropy-omega} yields Talagrand's inequality  \cref{eq:Talagrand}.
\paragraph{Marton's universal transport inequalities}
Choosing $\rho = d_H$, $\alpha(x)=x^2$, $x\geq0$ and $a=4$ in  \cref{equation:Marton} gives Marton's inequality  \cref{eq:Marton}, which holds true for any probability measure $\gamma$.
Actually, Marton's inequality \cref{eq:Marton} can be slightly improved.
Dembo \cite{Dembo} shows that any probability measure $\gamma$ satisfies a family of inequalities \cref{equation:Marton} with respect to the Hamming distance $d_{H}$, namely:
\begin{equation*}
  \widetilde{\cT}_{\alpha_{t},d_{H}}(\nu_{2} | \nu_{1}) \leq \frac{1}{t} \cH(\nu_{1} | \gamma) + \frac{1}{1-t} \cH(\nu_{2} | \gamma), \qquad \forall t \in [0,1];
\end{equation*}
where
\begin{equation}\label{eq:Dembo}
  \alpha_{t}(u) = \frac{t(1-u) \log (1-u) - (1 - tu) \log (1 - tu)}{t(1-t)}, \qquad \forall u\in [0,1].
\end{equation}
For $t = 0$ or $t = 1$, $\alpha_{t}$ is defined by taking the corresponding limit in the definition above, namely:
\begin{align*}
  & \alpha_{0}(u) = (1-u) \log(1-u) + u; \\
  & \alpha_{1}(u) = -u - \log(1-u).
\end{align*}
We easily check that $\alpha_{t}(u) \geq \alpha_{0}(u) \geq \frac{u^2}{2}$, $u\in [0,1]$.
In particular, with $t = \frac{1}{2}$, Dembo's result gives back Marton's inequality \cref{eq:Marton}.

\subsection{Stability by tensorization, by push-forward, and by approximation}\label{section:ancillary}
We now state the three main operations that enable us to obtain transport inequalities for mixed binomial processes: the tensorization, push-forward, and strong approximations.
Importantly for us, transport-entropy inequalities are closed under these operations.

Transport-entropy inequalities enjoy the following well known tensorization property~\cite[Theorem 4.11]{GRSTGeneral}.
\begin{proposition}\label{proposition:tensorization}
  Consider, for $i = 1, \dots, n$, costs functions $c_{i} \colon E_{i} \times \CP(E_{i}) \to [0,\infty]$ for some Polish spaces $E_{i}$.
  Suppose that the functions $c_i$ are convex with respect to their second variables.
  If, for all $1\leq i \leq n$, $\gamma_{i} \in \CP(E_{i})$ satisfies \hyperref[equation:transport_entropy-c]{\textnormal{($\lT_{c_{i}}$)}} with constant $a_{1}, a_{2} > 0$, then $\gamma_{1} \otimes \dots \otimes \gamma_{n}$ satisfies \hyperref[equation:transport_entropy-c]{\textnormal{($\lT_{\bar{c}}$)}}  (with the same constant $a_{1}, a_{2}$) where 
  \begin{equation}
    \bar{c}(x, p) = \sum_{i=1}^{n} c_{i}(x_{i}, p_{i}),
  \end{equation}
  where $x = (x_{1}, \dots, x_{n}) \in E_{1} \times \dots \times E_{n}$ and $p \in \CP(\bigtimes E_{i})$ has $i$-th marginal $p_{i}$.
\end{proposition}
\begin{remark}
  The reference \cite{GRSTGeneral} proves the theorem only for the case $a_{1}, a_{2} < \infty$ but if $a_{1} = 1$ (resp.\ $a_{2} = \infty$) their proof still works, replacing $\nu'$ (resp.\ $\nu$) by $\mu^{2}$.
\end{remark}

Now let us turn to the stability by push-forward. 
Let us fix two Polish spaces $X$ and $Y$, and a measurable map $S \colon X \to Y$.
Recall that for $\gamma \in \CP(X)$ the \emph{push forward} of $\gamma$ by $S$ is the element of $\CP(Y)$ defined by $S_{\#} \gamma(A) = \gamma(S^{-1}A)$, for all $A \in \gB(Y)$.
Given a cost $c \colon X \times \CP(X) \to [0,\infty]$ we define its \emph{push forward} by
\begin{equation}\label{equation:cbar}
  S_{\#} c(y,q) = \inf \{c(x,p) : S (x) = y, \ S_{\#} p=q\},
\end{equation}
for all $y \in Y$ and $q \in \CP(Y)$.
Likewise, for $\omega \colon X \times X \to [0,\infty]$, we write
\begin{equation}\label{equation:wbar}
    S_{\#} \omega(y_{1},y_{2}) = \inf \left\{ \omega(x_{1},x_{2}):  S(x_{1})= y_{1},\, S(x_{2})=y_{2} \right\},\qquad y_1,y_2 \in Y.
\end{equation}
We now show that the push-forward preserves transport-entropy inequalities.

\begin{proposition}\label{proposition:contraction}
  Suppose that $\gamma \in \CP(X)$ satisfies~\cref{equation:transport_entropy-c} (with constants $a_{1}, a_{2} > 0$) for some convex cost function $c \colon X \times \CP(X) \to \dR_{+}$.
  Then, the probability measure $S_{\#}  \gamma$ satisfies \hyperref[equation:transport_entropy-c]{\textnormal{($\lT_{S_{\#} c}$)}} with the same constants $a_{1}, a_{2}$.
  In particular, if $\gamma$ satisfies \cref{equation:transport_entropy-omega} for some $\omega \colon X \times X \to [0,\infty]$, then $S_{\#} \gamma$ satisfies \hyperref[equation:transport_entropy-omega]{\textnormal{($\lT_{S_{\#} \omega}$)}}.
\end{proposition}
This property is rather classical for transport inequalities of the form $(\lT_{\omega})$; it goes back at least to \cite{Maurey91} (in the dual language of infimum convolution inequalities). See \cite{Gozlan07, Gozlan12, GRSTGeneral} for applications of this transfer principle.  However, at this level of generality, \cref{proposition:contraction} is new.

\begin{proof}
  For short, we write $\bar{\gamma} = S_{\#}\gamma$ and $\bar{c} = S_{\#}c$.
Let $\bar{\nu}_1,\bar{\nu}_2 \in \CP(Y)$ be such that $\cH(\bar{\nu}_1|\bar{\gamma})<\infty$ and $\cH(\bar{\nu}_2|\bar{\gamma})<\infty$. 
Let $\bar{h}_1$ be the density of $\bar{\nu}_1$ with respect to $\bar{\gamma}$; then for $f \in \CC_{b}(Y)$
\begin{equation}
  \int f\dd\bar{\nu}_1 = \int f \bar{h}_1\, d\bar{\gamma} = \int f(S) \bar{h}_1(S)\dd\gamma = \int f(S)\dd\nu_1,
\end{equation}
where  $\dd\nu_1 = \bar{h}_1(S)\dd\gamma$.
Therefore, there exists at least one probability measure $\nu_1$ on $X$ such that $\bar{\nu}_1 = S_{\#}  \nu_1$.
On the other hand, 
\begin{equation}
  \cH(\bar{\nu}_1 | \bar{\gamma}) = \int \bar{h}_1\log \bar{h}_1 \dd\bar{\gamma} = \int \bar{h}_1(S)\log \bar{h}_1(S) \dd\gamma = \cH(\nu_1|\gamma).
\end{equation}
Let us consider the function $\overline{\cT}_c(\,\cdot\,|\,\cdot\,) \colon \CP(X) \times \CP(X) \to [0,\infty]$ defined by
\begin{equation}
  \overline{\cT}_c(\bar{\nu}_1|\bar{\nu}_2) = \inf \left\{\cT_c(\nu_1|\nu_2): \bar{\nu}_1=S_{\#} \nu_1\ \text{and}\ \bar{\nu}_2 = S_{\#} \nu_2\right\}.
\end{equation}
According to what precedes, for all $\bar{\nu}_i$, $i=1,2$, such that $\cH(\bar{\nu}_i|\bar{\gamma})<\infty$, there exist $\nu_i$, $i=1,2$, on $\CP(X)$ such that $\bar{\nu}_i = S_{\#}  \nu_i$ and so
\begin{equation}
  \overline{\cT}_c(\bar{\nu}_1|\bar{\nu}_2) \leq \cT_c(\nu_1|\nu_2) \leq a_{1} \cH(\nu_1|\gamma) + a_{2} \cH(\nu_2|\gamma) = a_{1} \cH(\bar{\nu}_1|\bar{\gamma}) + a_{2} \cH(\bar{\nu}_2|\bar{\gamma}).
\end{equation}
Now let us prove that
\begin{equation}\label{eq:comparison}
  \overline{\cT}_c(\bar{\nu}_1|\bar{\nu}_2) \geq  \cT_{\bar{c}}(\bar{\nu}_1|\bar{\nu}_2).
\end{equation}
Let $\bar{\nu}_1,\bar{\nu}_2$ such that $\cH(\bar{\nu}_i|\bar{\gamma})<+\infty$, $i=1,2$; there exist $\nu_1,\nu_2$ such that $\bar{\nu}_i=S_{\#} \nu_i$.
Let $p$ be a kernel such that $\nu_1 = \nu_2p$. Equivalently, there exists a pair of random variables $(X_1,X_2)$ with $X_2 \sim \nu_2$ and $\mathrm{Law}(X_1|X_2=x) = p_x$. Consider $Y_1 = S(X_1)$ and $Y_2=S(X_2)$; for all bounded continuous functions $f_1,f_2$ on $Y$ it holds
\begin{align*}
  \dE [f_1(Y_1)f_2(Y_2)] &= \dE[f_1(S(X_1))f_2(S(X_2))] = \dE\left[ \int f_1(S(x_1))\dd p_{X_2}(x_1) f_2(S(X_2))\right]\\
                         & = \dE\left[ \int f_1(y_1)\dd(S_{\#} p_{X_2})(y_1) f_2(S(X_2))\right]\\
                         & = \dE\left[ \dE\left[\int f_1(y_1)\dd(S_{\#} p_{X_2})(y_1) | S(X_2)\right] f_2(S(X_2))\right]
\end{align*}
Consider a regular conditional probability $k$ for $\mathrm{Law}(X_2|Y_2)$; then
\begin{equation}
  \begin{split}
    \dE [f_1(Y_1)f_2(Y_2)] &= \dE\left[ \int\left(\int f_1(y_1)\dd(S_{\#} p_{x_2})(y_1) \right) k_{Y_2}(dx_2) f_2(Y_2)\right] \\
                           &=\iint f_1(y_1)f_2(y_2)\bar{p}_{y_2}(dy_1)\,\bar{\nu}_2(dy_2),
  \end{split}
\end{equation}
with 
\begin{equation}
  \bar{p}_{y_2} = \int (S_{\#} p_{x_2})\,k_{y_2}(dx_2).
\end{equation}
This proves that $\bar{\nu}_2\bar{p}=\nu_1$.
\begin{align*}
  \int c(x_2,p_{x_2})\,\nu_2(dx_2) &\geq \int \bar{c} (S(x_2), S_{\#} p_{x_2})\,\nu_2(dx_2) \\
                                   & = \iint \bar{c} (S(x_2), S_{\#} p_{x_2})\,k_{y_2}(dx_2)\bar{\nu}_2(dy_2)\\
                                   & = \iint \bar{c} (y_2, S_{\#} p_{x_2})\,k_{y_2}(dx_2)\bar{\nu}_2(dy_2)\\
                                   & \geq \int \bar{c} \left(y_2, \int S_{\#} p_{x_2}\,k_{y_2}\right)\,\bar{\nu}_2(dy_2)\\
                                   & \geq \cT_{\bar{c}} (\bar{\nu}_1|\bar{\nu}_2),
\end{align*}
where the first inequality comes from the definition of $\bar{c}$, the third is a consequence of the fact that $S(x_2) = y_2$ for $k_{y_2}$ almost all $x_2$, and the fourth follows from the convexity of $\bar{c}$ (which is itself a simple consequence of the convexity of $c$).
Therefore, taking the infimum over $p$ yields to $\cT_c(\nu_1|\nu_2) \geq \cT_{\bar{c}}(\bar{\nu}_1 | \bar{\nu}_2)$.
Taking the infimum over all $\nu_1,\nu_2$ such that $S_{\#} \nu_i = \bar{\nu}_i$ finally gives~\cref{eq:comparison} and completes the proof.
\end{proof}

In applications, one often approximates a Poisson point process of interest by a simpler point process for which one can establish a transport-entropy inequality.
The following general lemma allows us to transfer a transport-entropy inequality from the simpler process to the other process.
\begin{lemma}\label{lemma:stability_Tc_approximation}
  Let $E$ be a Polish space and let $c \colon E \times \CP(E) \to \dR_{+}$ be a lower semi-continuous cost.
  Let $\{\gamma_{n};\; n \in \dN \} \cup \{\gamma\}$ such that for all $A \in \gB(E)$:
  \begin{equation*}
    \gamma_{n}(A) \tto{} \gamma(A).
  \end{equation*}
  If for all $n \in \dN$, $\gamma_{n}$ satisfies \cref{equation:transport_entropy-c} with constants $a_{1,n}, a_{2,n} >0$, then $\gamma$ satisfies \cref{equation:transport_entropy-c} with the constants $a_{i} = \liminf_{n} a_{i,n}$, $i=1,2.$
\end{lemma}
\begin{proof}
Thanks to the strong convergence, for any $f \colon E \to \dR$ bounded and measurable, it holds
\begin{equation}\label{eq:convergence}
  \int_{E} f(x) \gamma_{n}(\dd x) \tto{}  \int_{E} f(x) \gamma(\dd x).
\end{equation}
Take $\nu^{1},\nu^{2} \in \CP(E)$, absolutely continuous with respect to $\gamma$ with \emph{bounded} respective densities $f_{1}$ and $f_{2}$.
Let us define for all $i = 1,2$ and for all Borel set $A \subset E$,
\begin{equation*}
\nu^{i}_{n}(A) = \frac{\int_{A} f_{i}(\chi) \gamma_{n}(\dd \chi)}{\int f_{i}(x) \gamma_{n}(\dd x)}.
\end{equation*}
By assumption on $\gamma_n$, we have that 
\begin{equation}\label{eq:ITn}
  \cT_c(\nu^{1}_n|\nu^{2}_n) \leq a_{1,n} \cH(\nu^{1}_{n} | \gamma_{n}) + a_{2,n} \cH(\nu^{2}_{n} | \gamma_{n}).
\end{equation}
Using \cref{eq:convergence}, we see that $\nu^{i}_n \tto{} \nu^{i}$ weakly (and even setwise).
Since, by assumption, the cost function $c$ is lower semi-continuous, by \cref{theorem:lsc_transport_cost}, $\cT_{c} \colon \CP(E) \times \CP(E) \to [0,\infty]$ is also lower semi-continuous, and so 
\begin{equation*}
  \liminf_{n \to \infty} \cT_{c}(\nu^{1}_{n}|\nu^{2}_{n}) \geq \cT_{c}(\nu^{1}|\nu^{2}).
\end{equation*}
On the other hand, using \eqref{eq:convergence}, we see that for $i=1,2$
\begin{equation*}
  \cH(\nu^{i}_{n} | \gamma_{n}) = \frac{\gamma_{n}(f_{i} \log f_{i})}{\gamma_n(f_i)} -  \log \gamma_{n}(f_{i}) \tto{} \gamma(f_i \log f_i) = \cH(\nu^{i} | \gamma).
\end{equation*}
Letting $n \to \infty$ in \cref{eq:ITn} thus completes the proof in the case of bounded densities.
The general case follows by standard approximations arguments.
\end{proof}

\section{Transport inequalities for random point processes}\label{section:transport_ineq_point_process}
In this section, we establish transport-entropy inequalities for a mixed binomial point process of the form \cref{eq:mixed-binom} provided its sampling measure itself satisfies some transport-entropy inequalities.
We essentially base our proofs on the material of \cref{section:ancillary}.
The section is organized as follows: first we present some generalities about point processes, then we establish transport-entropy inequalities for mixed binomial point processes under the assumption that the sampling measure satisfies an inequality of the form \cref{equation:transport_entropy-omega} (a case which covers in particular our \cref{thm:Talagrand}), and finally we consider the case where the sampling measure satisfies an inequality of the form  \cref{equation:Marton} (which includes our \cref{thm:Marton}).

\subsection{Random point processes}\label{section:reminders_point_processes}
From now on, we fix a complete separable metric space $(Z,d)$ that we regard as our reference state space.
Recall the notation $\CM_{\dN}(Z)$ (resp.\ $\CM_{\bar{\dN}}(Z)$) introduced in \cref{eq:configuration} for the space of finite configurations (resp.\ configurations); equipped with the weak (resp.\ vague) topology (introduced in \cref{section:topology}), $\CM_{\dN}(Z)$ (resp.\ $\CM_{\bar{\dN}}(Z)$) is a Polish space according to \cref{lem:configuration}.
We always equip the space $\CM_{\dN}(Z)$ and $\CM_{\bar{\dN}}(Z)$ with their Borel $\sigma$-field (without further mention).
For $n\in \dN$, we also conveniently write $\CM_n(Z) = \{\xi \in \CM_{\dN}(Z) : \xi(Z) = n\}$.
This space is a closed subset of $\CM_{\dN}(Z)$ and is therefore also Polish when equipped with the weak topology.

A \emph{point process} (resp.\ \emph{finite point process}) on $Z$ is a $\CM_{\bar{\dN}}(Z)$-valued (resp.\ $\CM_{\dN}(Z)$-valued) random variable.
In this paper, we focus, on the one hand, on mixed binomial process defined in \cref{eq:mixed-binom}.
For reader's convenience, we recall that this is the class of finite point processes of the form $\eta = \sum_{i=1}^N \delta_{X_i}$ where $(X_i)_{i\geq 1}$ is an i.i.d sequence of law $\mu \in \CP(Z)$ independent of the $\dN$-valued random variable $N$ whose law is denoted by $\kappa$.
We also recall that the law of $\eta$ is denoted $B_{\mu,\kappa}$. When $\kappa = \delta_n$, where $n\in \dN$, we write $B_{\mu,n} \in \CP(\CM_n(Z))$ instead of $B_{\mu,\delta_n}$. A process $\eta \sim B_{\mu,n}$ is called a \emph{binomial process of size $n$}.

On the other hand, we consider the important class of \emph{Poisson point processes}.
Given a measure $\nu$ on $Z$, we say that $\eta$ is a \emph{Poisson point process with intensity measure $\nu$} whenever: for all pairwise disjoint measurable sets $A_{1}, \dots, A_{l} \in \gZ$, with $\nu(A_{i}) < \infty$ for all $i = 1,\dots, l$, the random vector $(\eta(A_{1}), \dots, \eta(A_{l}))$ is a vector of independent Poisson random variables with mean $(\nu(A_{1}), \dots, \nu(A_{l}))$.
Provided such process exists, the intensity measure $\nu$ characterizes the law of such a process; we write $\Pi_{\nu}$ to denote this law.
Existence of Poisson point process on arbitrary state space and arbitrary reference measure is in general a non-trivial fact.
However, if $\nu \in \CM_{b}(Z)$, one can easily check that a mixed binomial process $\eta$ constructed using $(X_i)_{i\geq1}$ of common law $\mu = \frac{\nu}{\nu(Z)}$ and a random variable $N$ with a Poisson distribution of mean $\nu(Z)$ is a Poisson process with intensity measure $\nu$. In other words: for all $\nu \in \CM_{b}(Z)$,
\begin{equation}
  \Pi_{\nu} = B_{\mu,\kappa},
\end{equation}
where $\mu = \frac{\nu}{\nu(Z)} \in \CP(Z)$ and $\kappa\in \CP(\dN)$ is the Poisson distribution of mean $\nu(Z)$.
By a simple gluing procedure~\cite[Theorem 3.6]{LastPenrose}, one can also construct Poisson point processes with $\sigma$-finite intensity measure.
More precisely, writing $\nu = \sum_{i=1}^{\infty} \nu_{n}$ with $\nu_{n} \in \CM_{b}(Z)$, we have that
\begin{equation*}
  \Pi_{\nu} = \Pi_{\nu_{1}} \ast \Pi_{\nu_{2}} \ast \dots,
\end{equation*}
where $\ast$ is the convolution product.
In general, $\Pi_{\nu} \in \CP(\CM_{\bar{\dN}}(Z))$; $\Pi_{\nu} \in \CP(\CM_{\dN}(Z))$ if and only if $\nu \in \CM_{b}(Z)$; and $\Pi_{\nu} \in \CP(\CM_{\bar{\dN}}(Z) \setminus \CM_{\dN}(Z))$ if and only if $\nu(Z) = \infty$.

\subsection{Transport-entropy inequalities when the sampling measure satisfies a Talagrand type inequality}\label{sec:Talagrand}
In this section, we derive some transport-entropy inequalities of the type \cref{equation:transport_entropy-omega} on $\CM_{\dN}(Z)$, where $\omega$ is of the form $\cT_{\rho} \colon \CM_{\dN}(Z) \times \CM_{\dN}(Z) \to [0,\infty]$ for some $\rho \colon Z \times Z \to [0,\infty]$.
We denote by $\dT_\rho$ the transport cost $\cT_{\cT_\rho}$, given explicitly by
\begin{equation*}
  \dT_\rho (\Pi_1,\Pi_2) = \inf \dE\left[\cT_\rho(\xi_1,\xi_2)\right],\qquad \Pi_1,\Pi_2 \in \CP(\CM_{\dN}(Z)),
\end{equation*}
where the infimum runs over the couples $(\xi_1,\xi_2)$ of point processes such that $\xi_1 \sim \Pi_1$ and $\xi_2 \sim \Pi_2$.
As already observed in \cref{section:intoduction}, the condition $\dT_\rho (\Pi_1,\Pi_2) <+\infty$ is very strong and it implies in particular that there exists a coupling $(\xi_1,\xi_2)$ of $\Pi_1,\Pi_2$ such that $\xi_1(Z)=\xi_2(Z)$ almost surely.
In particular, the finiteness of $\dT_{\rho}(\Pi_{1}, \Pi_{2})$ implies that the law of the total mass is the same under $\Pi_1$ and $\Pi_2$. 
\begin{theorem}\label{theorem:Talagrand_binomial}
  Let $\kappa \in \CP(\dN)$, $\mu \in \CP(Z)$, and $\rho \colon Z \times Z \to [0,\infty]$ be lower semi-continuous.
  Assume $\mu \in \CP(Z)$ satisfies \hyperref[equation:transport_entropy-omega]{\textnormal{($\lT_{\rho}$)}} with constants $a_{1}, a_{2} > 0$.
  Then, for all $\Pi_{1},\Pi_{2} \in \CP(\CM_{\dN}(Z))$ such that $\dT_{\rho}(\Pi_{1}, \Pi_{2}) < \infty$, it holds
  \begin{equation}\label{equation:Talagrand_binomial}
    \dT_{\rho}(\Pi_{1}, \Pi_{2}) + (a_{1} + a_{2}) \cH(\lambda|\kappa) \leq a_{1} \cH(\Pi_{1} | B_{\mu, \kappa}) + a_{2} \cH(\Pi_{2} | B_{\mu, \kappa}),
  \end{equation}
where $\lambda \in \CP(\dN)$ is such that for all $i\in \{1,2\}$, $\Pi_i( \{\eta \in \CM_{\dN}(Z) : \eta(Z) \in A\}) = \lambda(A)$, for all Borel $A \subset \dN$.
\end{theorem}
Note that the preceding result could be alternatively stated on the bigger state space $\CM_b(Z)$ instead of $\CM_{\dN}(Z)$. This actually does not make any difference since $ \cH(\Pi | B_{\mu, \kappa}) <+\infty$ implies that $\Pi \in \CM_{\dN}(Z)$.
Specifying $\rho = d^2$ in \cref{theorem:Talagrand_binomial} gives \cref{thm:Talagrand}.

\begin{remark}
  Let us comment on \cref{equation:Talagrand_binomial}.
  
  This inequality is a transport-entropy inequality \cref{equation:transport_entropy-omega}, where $\omega = \cT_{\rho}$ (as introduced in \cref{sec:defTI}) with the \emph{additional constraint} of finiteness of the transport cost: $\dT_{\rho}(\Pi_{1}, \Pi_{2})<+\infty$.
This condition is crucial since otherwise it would be possible to have a finite right hand side and an infinite left hand side. For instance, take $\mu\in \CP(Z)$ satisfying \hyperref[equation:transport_entropy-omega]{\textnormal{($\lT_{\rho}$)}} with some constant $a > 0$, take $n_{1} \ne n_{2}$ and consider $B_{\mu,n_{i}}$ the binomial point process of size $n_i$. Then $\dT_{\rho}(B_{\mu,n_{1}}, B_{\mu,n_{2}}) = \infty$. On the other hand, denoting by $\Pi_\mu$ the law of the Poisson point process with intensity measure $\mu$ and taking $\eta\sim \Pi_\mu$, we easily see that $B_{\mu,n_i} = \mathrm{Law} (\eta | \eta(Z)=n_i)$. Therefore, $B_{\mu,n_i}$ is absolutely continuous with respect to $\Pi_{\mu}$ with Radon-Nikodym derivative given by $\frac{1}{\dP(\eta(Z) = n_i )} \1_{\{\eta(Z)=n_i\}}$ and, thus, $\cH(B_{\mu,n_{i}} | \Pi_{\mu})= - \log \dP(\eta(Z)=n_i) = -\log \frac{e^{-1}}{n_i!}<+\infty$.
\end{remark}

\subsubsection{Binomial process with deterministic size}
As a first step, we  prove \cref{theorem:Talagrand_binomial} in the particular case of binomial point process of fixed size $n \in \dN$.
\begin{proposition}\label{prop:binom}
Let $n\in \dN$, $\mu \in \CP(Z)$ and $\rho \colon Z \times Z \to [0,\infty]$ be lower semi-continuous.
Assume $\mu \in \CP(Z)$ satisfies \hyperref[equation:transport_entropy-omega]{\textnormal{($\lT_{\rho}$)}} with constants $a_{1}, a_{2} > 0$.
  Then, for all $\Pi_{1},\Pi_{2} \in \CP(\CM_{n}(Z))$ such that $\dT_{\rho}(\Pi_{1}, \Pi_{2}) < \infty$, it holds
\begin{equation*}
  \dT_{\rho}(\Pi_{1}, \Pi_{2})  \leq a_{1} \cH(\Pi_{1} | B_{\mu, n}) + a_{2} \cH(\Pi_{2} | B_{\mu, n}).
\end{equation*}
\end{proposition}
\begin{proof}
For $n=0$ there is nothing to prove. Fix $n \geq 1$; by the tensorization property \cref{proposition:tensorization}, the probability measure $\mu^{\otimes n} \in \CP(Z^n)$ satisfies the transport-entropy inequality  \hyperref[equation:transport_entropy-omega]{\textnormal{($\lT_{\rho^n}$)}} with constant $a$ and with the cost function $\rho^n$ defined by
\begin{equation*}
\rho^n(x,y) = \sum_{i=1}^n \rho(x_i,y_i),\qquad x,y \in Z^n.
\end{equation*}
Now, consider the map $S^n \colon Z^{n} \to \CM_{n}(Z)$ defined by 
\begin{equation}\label{eq:Sn}
  S^n(z) = \sum_{i=1}^{n} \delta_{z_{i}},\qquad \forall z = (z_{1}, \dots, z_{n}) \in Z^n.
\end{equation}
By construction, $S^n_{\#} \mu^{\otimes n} = B_{\mu,n}$. Furthermore, for any $\xi_1,\xi_2 \in \CM_n(Z)$, it follows from \cref{prop:birkhoff} that
\begin{align*}
\cT_{\rho}(\xi_2|\xi_1) &= \inf \left\{\sum_{i=1}^n \rho(x_i,y_i) : \xi_1 = S^n(x), \xi_2 = S^n(y)\right\}\\
& =S^n_\# \rho^n,
\end{align*}
using the notation introduced in  \cref{equation:wbar}.
Finally, we obtain from \cref{proposition:contraction} that $B_{\mu,n}$ satisfies a transport-entropy inequality on $\CM_{n}(Z)$ with respect to the cost function $\cT_{\rho}$ (and with the same constants $a_{1},a_{2}$).
\end{proof}

\subsubsection{Binomial processes with general size}
Now we are ready to prove \cref{theorem:Talagrand_binomial} in the general case.
\begin{proof}[Proof of \cref{theorem:Talagrand_binomial}]
Let $\Pi_{1},\Pi_{2} \in \CP(\CM_{\dN}(Z))$ such that $\dT_\rho(\Pi_{1}, \Pi_{2}) < \infty$ and $\cH(\Pi_i|B_{\mu,\kappa})<+\infty$, $i=1,2$. Recall that, according to \cref{lem:configuration}, $\CM_{\dN}(Z)$ endowed with the weak topology is a Polish space. According to \cref{theorem:lsc_transport_cost}, the transport cost $\cT_{\rho} \colon \CM_{\dN}(Z) \times  \CM_{\dN}(Z) \to [0,\infty]$ is lower semi-continuous.
Therefore, according to \cite[Theorem 4.1]{Villani09} (or \cref{theorem:existence_weak_coupling}), there exists an optimal coupling for $\dT_\rho(\Pi_{1},\Pi_{2})$; we write $(\xi_{1}, \xi_{2})$ for such an optimal coupling.
From the finiteness assumption of the transport distance, we have that $\xi_{1}(Z) = \xi_{2}(Z) := K$ almost surely.
We denote by $\lambda$ the law of $K$.
For $n \in \dN$ such that $\lambda(n)>0$, we write
\begin{equation*}
  \Pi^{n} = \mathrm{Law}((\xi_{1}, \xi_{2}) | K = n),
\end{equation*}
we consider $(\xi_{1}^{n}, \xi_{2}^{n}) \sim \Pi^{n}$, and we write $\Pi_{1}^{n}$ (resp.\ $\Pi_{2}^{n}$) for the law of $\xi_{1}^{n}$ (resp.\ $\xi_{2}^{n}$), that is $\Pi_{1}^{n}$ and $\Pi_{2}^{n}$ are the marginals of $\Pi^{n}$. Applying \cref{thm:stability}, we see that $(\xi_{1}^{n},\xi_{2}^{n})$ is an optimal coupling for $\dT_\rho(\Pi_{1}^{n}, \Pi_{2}^{n})$.
By construction, we have that
\begin{equation*}
  \begin{split}
    \dT_\rho(\Pi_{1},\Pi_{2}) &= \dE \left[\cT_{\rho}(\xi_{1},\xi_{2})\right] \\
                       &= \sum_{n \in \dN} \dE \left[ \cT_{\rho}(\xi_{1}, \xi_{2}) | K = n \right] \dP(K = n) \\
                       &= \sum_{n \in \dN} \dE \left[\cT_{\rho}(\xi_{1}^{n}, \xi_{2}^{n})\right] \dP(K=n) \\
                                   &= \sum_{n \in \dN} \dT_\rho(\Pi_{1}^{n}, \Pi_{2}^{n}) \dP(K=n).
  \end{split}
\end{equation*}
According to \cref{prop:binom}, we know that $B_{\mu,n}$ satisfies a transport-entropy inequality with the cost function $\cT_\rho$ on $\CM_n(Z)$ and constants $a_{1},a_{2}$. 
Thus
\begin{equation}\label{eq:ITpartial}
  \dT_\rho(\Pi_{1},\Pi_{2}) \leq \sum_{n\in \dN} ( a_{1} \cH(\Pi_{1}^{n}|B_{\mu,n}) + a_{2} \cH(\Pi_{2}^{n} | B_{\mu,n})) \dP(K=n).
\end{equation}
It follows from the chain rule for relative entropy (recalled below) that
\begin{equation}\label{equation:chain_rule_entropy}
  \cH(\Pi_{i} | B_{\mu,\kappa}) = \cH(\lambda|\kappa) + \sum_{n \in \dN} \cH(\Pi_{i}^{n} | B_{\mu,n}) \dP(K=n).
\end{equation}
Plugging  \cref{equation:chain_rule_entropy} into  \cref{eq:ITpartial} gives  \cref{equation:Talagrand_binomial} and completes the proof.

For the sake of completeness, we recall the proof of  \cref{equation:chain_rule_entropy}.
Let $\eta \sim B_{\mu,\kappa}$ with $N \sim \kappa$ (see  \cref{eq:mixed-binom}).
Recall that we assume that $\cH(\Pi_i | B_{\mu,\kappa})<+\infty$.
Denoting by $f_i \colon \CM_{\dN}(Z) \to \dR_{+}$ the Radon-Nikodym derivative of $\Pi_i$ with respect to $B_{\mu,\kappa}$, for all measurable set $A \subset \CM_{\dN}(Z)$, we have that
\begin{equation}
  \begin{split}
    \dP(\xi_{i}^{n} \in A) \dP(K = n) & = \dP(\xi_{i} \in A, K = n) \\
                                       &= \dE \left[\1_{\{\eta \in A \cap \CM_{n}(Z)\}} f_i(\eta) \right] \\
                                       &= \dE \left[\1_{\{N=n\}} \1_{\{\eta \in A\}} f_i(\eta)\right] \\
                                       &= \dP(N = n) \dE \left[ \1_{\{\eta \in A\}} f_i(\eta) | N= n \right].
  \end{split}
\end{equation}
This shows that if $\dP(K =n) > 0$ then $\dP(N = n) >0$ and for such $n$ we have that
\begin{equation}
  \frac{\dd \Pi_{i}^{n}}{\dd B_{\mu,n}} = \frac{\dd \Pi_{i}}{\dd B_{\mu,\kappa}} \frac{\dP(N=n)}{\dP(K=n)} \1_{\{N=n\}}.
\end{equation}
Hence,
\begin{equation}
\cH(\Pi_{i}^{n} | B_{\mu, n}) \dP(K=n) = \dE[\log f_i(\xi_{i}) | N=n] \dP(N=n) - \dP(K=n) \log \frac{\dP(K=n)}{\dP(N=n)}.
\end{equation}
Summing the previous expression for $n \in \dN$ yields   \cref{equation:chain_rule_entropy}.
\end{proof}

\subsection{Universal Marton-type inequalities for Poisson point processes}\label{sec:Marton}

\subsubsection{Deterministic size case}
Recall the definition of the partial transport cost $\cT_{\rho, 0}$ introduced at the end of \cref{section:transport_configuration_space}.
\begin{theorem}\label{theorem:Marton_binomial}
Let $n\in \dN$, $\mu \in \CP(Z)$.
Assume that $\mu$ satisfies \cref{equation:Marton} on $Z$ for $a_{1}, a_{2} > 0$ and for some lower semi-continuous cost function $\rho \colon Z \times Z \to [0,\infty]$ and $\alpha \colon [0,\infty] \to [0,\infty]$ convex.
Then, $B_{\mu,n}$ satisfies the transport-entropy inequality \cref{equation:transport_entropy-c} (with the same constants) on $\CM_{n}(Z)$ with the cost function $c \colon \CM_{n}(Z) \times \CP(\CM_{n}(Z)) \to \dR_{+}$ defined, for all $\xi \in \CM_{n}(Z)$ and $\Pi \in \CP(\CM_{n}(Z))$, by
  \begin{equation}\label{eq:Marton_binomial}
    c(\xi, \Pi)  = \int \alpha\left(\frac{1}{{\xi(x)}} \int \cT_{\rho, 0}(\chi, \xi(x)\delta_{x}) \Pi(\dd \chi)\right) \xi(\dd x).
  \end{equation}
\end{theorem}

Before proving \cref{theorem:Marton_binomial}, let us highlight a particular case.
Thanks to the universal transport inequalities by Marton and Dembo recalled in \cref{sec:defTI}, binomial processes always satisfy some transport-entropy inequality, as shown in the following result.
\begin{corollary}\label{theorem:transport_entropy_binomial}
  Let $t \in [0,1]$ and let $\alpha_{t}$ be the function defined by \cref{eq:Dembo}.
  Let $\mu \in \CP(Z)$ and $n \in \dN$.
  Then, $B_{\mu,n}$ satisfies a transport-entropy inequality \cref{equation:transport_entropy-c}, namely
  \begin{equation*}
    \cT_{c_{t}}(\Pi_{2} | \Pi_{1}) \leq \frac{1}{t} \cH(\Pi_{1} | B_{\mu,n}) + \frac{1}{1-t} \cH(\Pi_{2} | B_{\mu, n}),
  \end{equation*}
     where, for all $\xi \in \CM_{n}(Z)$ and $\Pi \in \CP(\CM_{n}(Z))$:
  \begin{equation}\label{equation:cost_binomial}
    c_{t}(\xi, \Pi) = \int \alpha_{t} \left(\int {\left[1 - \frac{\chi(x)}{\xi(x)}\right]}_{+} \Pi(\dd \chi)\right) \xi(\dd x).
  \end{equation}
\end{corollary}
Since $\alpha_{\frac{1}{2}}(u) \geq \frac{u^2}{2}$, $u\in [0,1]$, when $t=\frac{1}{2}$, the same conclusion holds with $\frac{u^{2}}{2}$ replacing $\alpha_{t}$ in \cref{equation:cost_binomial}.
\begin{proof}[Proof of  \cref{theorem:transport_entropy_binomial}]
  According to \cref{sec:defTI}, the probability measure $\mu$ satisfies the inequality \hyperref[equation:Marton]{\textnormal{($\widetilde{\lT}_{\alpha_{t}, d_{H}}$)}} with $\alpha_{t}$ given by \cref{eq:Dembo} and the constants $a_{1} = \frac{1}{t}$ and $a_{2} = \frac{1}{1-t}$.
  Moreover, if $\chi, \xi \in \CM_{n}(Z)$, then using that $\xi(x) \leq \xi(Z) = \chi(Z)$, we find that
  \begin{equation}
    \cT_{d_{H},0}(\chi, \xi(x)\delta_{x}) = {(\xi(x) - \chi(x))}_{+}.
  \end{equation}
Applying  \cref{theorem:Marton_binomial} concludes the proof. 
\end{proof}

\begin{proof}[Proof of \cref{theorem:Marton_binomial}]
  We follow the proof of \cref{prop:binom}. According to \cref{proposition:tensorization}, the probability measure $\mu^{\otimes n}$ satisfies the transport inequality \cref{equation:transport_entropy-c} on $Z^{n}$ with constants $a_{1}, a_{2}$ with respect to the cost function $c$ defined by
\begin{equation}
  c(x,p) =  \sum_{i=1}^{n} \alpha\left(\int \rho(x_{i}, y) p_{i}(\dd y)\right),\qquad x \in Z^{n},\, p \in \CP(Z^{n}).
\end{equation}
The probability measure $B_{\mu,n}$ is the push-forward of $\mu^{\otimes n}$ by the map $S^n$ defined by  \cref{eq:Sn}.
Therefore, according to \cref{proposition:contraction}, we see that $B_{\mu,n}$ satisfies the transport-entropy inequality with the cost function $\bar{c}$ defined for all $(\xi,\Pi)\in \CM_{n}(Z) \times \CP(\CM_{n}(Z))$ by
\begin{equation}
  \bar{c}(\xi,\Pi) = \inf \left\{c(x , p) : (x,p)\ \text{such that}\ S^n(x) = \xi,\ S^n_{\#} p=\Pi \right\}.
\end{equation}
In other words, if $\xi = \sum_{i=1}^{n} \delta_{a_{i}}$, then
\begin{equation}
\bar{c}(\xi,\Pi) =  \inf\left\{\sum_{i=1}^{n} \alpha\left(\dE \left[\rho(Y_{i}, x_{i})\right]\right) \right\},
\end{equation}
where the infimum runs over all random vectors $Y=(Y_1,\ldots,Y_n)$ such that $\sum_{i=1}^{n} \delta_{Y_{i}} \sim \Pi$ (whose existence is given by \cite[Proposition 6.3]{LastPenrose}) and all $x=(x_1,\ldots,x_n) \in Z^{n}$ such that $\xi = \sum_{i=1}^{n} \delta_{x_{i}}$.
The constraint $\xi = \sum_{i=1}^{n} \delta_{x_{i}}$ determines the $x_{i}$'s up to permutation. Since the other constraint $\sum_{i=1}^{n} \delta_{Y_{i}} \sim \Pi$ is permutation invariant, one can assume without loss of generality that $a_i=x_i$ for all $i\in \{1,\ldots,n\}$.
In this way, we obtain that
\begin{equation}
  \bar{c}(\xi,\Pi) = \inf \left\{\sum_{i=1}^{n} \alpha\left(\dE\left[ \rho(Y_{i}, a_{i})\right]\right) : \sum_{i=1}^{n} \delta_{Y_{i}} \sim \Pi \right\}.
\end{equation}
Let us prove that $\bar{c}$ is bounded from below by the cost given in  \cref{eq:Marton_binomial}.
For all $a \in Z$, such that $\xi(a) > 0$, we define $I(a) = \{ i : a_{i} = a \}$.
Then, letting $\chi = \sum_{i=1}^{n} \delta_{Y_{i}}$, we obtain
\begin{align*}
    \sum_{i=1}^{n} \alpha(\dE\left[ \rho(Y_i, a_{i})\right] ) &=
    \sum_{a : \xi(a)>0} \xi(a) \frac{\sum_{i\in I(a)} \alpha(\dE\left[ \rho(Y_i, a_{i})\right] )}{\xi(a)}\\
                                                          &\geq \sum_{a : \xi(a)>0} \xi(a)\alpha\left(\dE\left[ \frac{\sum_{i \in I(a)}  \rho(Y_i, a)}{\xi(a)} \right]\right)\\
                                                          & \geq \sum_{a : \xi(a)>0} \xi(a)\alpha\left(\dE\left[ \frac{\cT_{\rho,0}(\chi,\xi(a)\delta_a)}{\xi(a)} \right]\right),
\end{align*}
where the first inequality comes from Jensen inequality and the second by definition of \cref{equation:partial_transport_cost}. This completes the proof.
\end{proof}

\subsubsection{Poisson case}
First, a process level ``binomial to Poisson argument'', adapted from \cite{ReitznerConcentration}, allows us to extend the conclusion of \cref{theorem:Marton_binomial} to Poisson processes with finite intensity measure.
Then we will use another approximation argument to pass from the case of finite intensity measure to the case of $\sigma$-finite intensity measure.

\begin{theorem}\label{theorem:transport_entropy_poisson}
  Let $\alpha_{t}$ be the function defined by \cref{eq:Dembo}.
  Let $\nu$ be a $\sigma$-finite Radon measure on $Z$.
  Then $\Pi_\nu$ satisfies a transport-entropy inequality \cref{equation:transport_entropy-c}, namely
  \begin{equation*}
    \cT_{c_{t}}(\Pi_{2} | \Pi_{1}) \leq \frac{1}{t} \cH(\Pi_{1} | \Pi_\nu) + \frac{1}{1-t} \cH(\Pi_{2} | \Pi_\nu),
  \end{equation*}
  where the cost function $c_{t} \colon \CM_{\bar{\dN}}(Z) \times \CP(\CM_{\bar{\dN}}(Z)) \to [0,\infty]$ is defined by
  \begin{equation}\label{equation:cost_poisson}
    c_{t}(\xi, \Pi) = \int \alpha_{t} \left(\int {\left[1 - \frac{\chi(x)}{\xi(x)}\right]}_{+} \Pi(\dd \chi)\right) \xi(\dd x),\qquad \xi \in \CM_{\bar{\dN}}(Z),\Pi \in \CP(\CM_{\bar{\dN}}(Z))
  \end{equation}
\end{theorem}
In particular, this result gives \cref{thm:Marton}. 
\begin{proof}
We first assume that there exists $\mu \in \CP(Z)$ and $\lambda>0$ such that $\nu = \lambda \mu$.
For $\mu \in \CP(Z)$, $n \in \dN$, and $t \in [0,1]$, the $t$-thinning of $B_{\mu,n}$ is the law denoted by $B_{\mu,n}^{(t)}$ of the point process obtained by removing each point of $\eta \sim B_{\mu,n}$ with probability $(1-t)$ independently of the others.
Formally, this family of point processes can be realized for instance as follows: for $t \in [0,1]$, set
\begin{equation*}
\tilde{\eta}^{(t)} = \sum_{i=1}^n \varepsilon_i \delta_{X_i}
\end{equation*}
with $X_1,\ldots,X_n$ i.i.d of law $\mu$ and $ \varepsilon_1,\ldots, \varepsilon_n$ i.i.d with a Bernoulli distribution of parameter $t \in [0,1]$. Then $B_{\mu,n}^{(t)}$ is by definition the law of $\tilde{\eta}^{(t)}$.
We conveniently use a slightly different construction taken from \cite{ReitznerConcentration}.
Let us add a cemetery point $\infty$ to $Z$ by considering $\hat{Z} = Z \cup \{\infty\}$ where $\infty \not \in Z$ and $\{\infty\}$ is isolated.
Define
\begin{equation*}
\hat{\eta}^{(t)} = \sum_{i=1}^n \delta_{\hat{X}_i^{(t)}},
\end{equation*}
where $\hat{X}_1^{(t)},\ldots,\hat{X}_n^{(t)}$ are i.i.d random variables with law $t \mu + (1-t) \delta_{\infty}$. Defining
\begin{equation*}
\eta^{(t)} = \hat{\eta}^{(t)}_{\,| Z} =  \sum_{i=1}^n \1_{\{\hat{X}_i^{(t)} \neq \infty\}}\delta_{\hat{X}_i^{(t)}},
\end{equation*}
we then easily see that $\eta^{(t)}$ and $\tilde{\eta}^{(t)}$ have the same law.

Let $\lambda > 0$, and $n$ sufficiently large so that $\lambda_{n} := \lambda /n \leq 1$; we write $\Pi_{n} = B_{\mu,n}^{(\lambda_{n})}$. This sequence $\Pi_n$ converges to $\Pi_\nu$, the law of the Poisson process of intensity measure $\nu := \lambda \mu$.
Indeed, according to \cite[Theorem 4.34]{Kallenberg}, one gets that for all Borel set $A \subset \CM_{\dN}(Z)$:
\begin{equation*}
\Pi_{n}(A) \tto{} \Pi_{\nu}(A).
\end{equation*}

Since according to \cref{lemma:c_sci} below, the cost function $c$ is lower semi-continuous, in view of the previous convergence and \cref{lemma:stability_Tc_approximation}, to conclude the proof in the case of finite intensity measure, we simply establish a transport-entropy inequality for $\Pi_n$ ($n \in \dN$).
Let $\hat{\mu}_{n} = \lambda_{n} \mu + (1-\lambda_{n}) \delta_{\infty}$; according to \cref{theorem:transport_entropy_binomial}, we see that $B_{\hat{\mu}_n,n} \in \CP(\CM_{n}(\hat{Z}))$ satisfies the transport-entropy inequality with constants $a_{1} = \frac{1}{t}$ and $a_{2} = \frac{1}{1-t}$, and with cost function $\hat{c} \colon \CM_{n}(\hat{Z}) \times \CP(\CM_{n}(\hat{Z})) \to \dR_+$ defined by
\begin{equation*}
  \hat{c}(\hat{\xi}, \hat{\Pi}) = \int \alpha_{t} \left(\int {\left[1 - \frac{\hat{\chi}(x)}{\hat{\xi}(x)}\right]}_{+} \hat{\Pi}(\dd \hat{\chi})\right) \hat{\xi}(\dd x),\qquad \hat{\xi} \in \CM_{n}(\hat{Z}),\hat{\Pi} \in \CP(\CM_{n}(\hat{Z})),
\end{equation*}
where $\alpha_{t}$ is the function defined in \cref{eq:Dembo}.
By construction, $\Pi_n$ is the push-forward of $B_{\hat{\mu}_n,n}$ under the map $S \colon \CM_{n}(\hat{Z}) \to \CM_{\dN}(Z)$ defined, for all $\hat{\xi} = \sum_{i=1}^n \delta_{\hat{x}_i}$, $\hat{x}_1,\ldots,\hat{x}_n \in \hat{Z}$, by
\begin{equation*}
S(\hat{\xi}) = \hat{\xi}_{\,| Z} = \sum_{i=1}^n \1_{\{\hat{x}_i \in Z\}}\delta_{\hat{x}_i}.
\end{equation*}
According to \cref{proposition:contraction}, it follows that $\Pi_n$ satisfies the transport-entropy inequality with the same constants and with the cost function $\bar{c} \colon \CM_{\dN}(Z) \times \CP(\CM_{\dN}(Z)) \to \dR_+$ defined by
\begin{equation*}
\bar{c} (\xi,\Pi) = \inf \left\{ c(\hat{\xi},\hat{\Pi}) : S(\hat{\xi}) = \xi, S_\#\hat{\Pi} = \Pi\right\}, \qquad \xi \in \CM_{\dN}(Z),\Pi \in \CP(\CM_{\dN}(Z)),
\end{equation*}
Now, if $S(\hat{\xi}) = \xi$ and $S_\#\hat{\Pi} = \Pi$, then
\begin{align*}
  \hat{c}(\hat{\xi}, \hat{\Pi}) & = \int_{\hat{Z}} \alpha_{t} \left(\int \left[1 - \frac{\hat{\chi}(x)}{\hat{\xi}(x)}\right]_{+} \hat{\Pi}(\dd \hat{\chi})\right) \hat{\xi}(\dd x)\\
                              & \geq \int_{Z} \alpha_{t} \left(\int \left[1 - \frac{S(\hat{\chi})(x)}{\xi(x)}\right]_{+} \hat{\Pi}(\dd \hat{\chi})\right) \xi(\dd x)\\
                              & = \int_{Z} \alpha_{t} \left(\int \left[1 - \frac{\chi(x)}{\xi(x)}\right]_{+} \Pi(\dd \chi)\right) \xi(\dd x)\\
& = c(\xi,\Pi),
\end{align*}
where, for short, we write $c =c_{t}$ for the cost function defined in \cref{equation:cost_poisson}.
This concludes the proof for Poisson point process with finite intensity measure.

Now, we assume that $\nu$ is $\sigma$-finite, and we let $(Z_{n})$ be an increasing measurable exhaustion of $Z$ in measurable sets of finite measure.
Let $\Pi_{n} = \Pi_{\nu_{n}}$ where $\nu_{n} = \nu(\cdot \cap Z_{n})$.
We have that for all $m \in \dN$,
\begin{equation*}
  \nu_{n}(\cdot \cap Z_{m}) \tto{TV} \nu(\cdot \cap Z_{m}).
\end{equation*}
According to \cite[Theorem 4.33]{Kallenberg}, we find that $\Pi_{n} \tto{TV} \Pi$.
By \cref{lemma:c_sci}, $c_{t}$ is lower semi-continuous, and by the previous analysis $\Pi_{n}$ satisfies \cref{equation:transport_entropy-c} for all $n \in \dN$.
Thus, by the strong convergence of $\Pi_{n}$ to $\Pi$ and \cref{lemma:stability_Tc_approximation}, we find that $\Pi_{\nu}$ satisfies \cref{equation:transport_entropy-c}.
\end{proof}

Now we turn to the proof of the lower semi-continuity of $c$.
\begin{lemma}\label{lemma:c_sci}
  Let $\alpha \colon\dR \to [0,\infty]$ be a any lower semi-continuous function.
  The function $c \colon \CM_{\bar{\dN}}(Z) \times \CP(\CM_{\bar{\dN}}(Z))$ defined by
  \begin{equation*}
    c(\xi, \Pi) = \int \alpha \left(\int {\left[1-\frac{\chi(x)}{\xi(x)}\right]}_{+} \Pi(\dd \chi) \right) \xi(\dd x)
  \end{equation*}
  is lower semi-continuous.
  Moreover, the restriction of $c$ to $\CM_{\dN}(Z) \times \CP(\CM_{\dN}(Z))$ is also lower semi-continuous.
\end{lemma}

\begin{proof}
  We will first show that $c \colon \CM_{\dN}(Z) \times \CP(\CM_{\dN}(Z)) \to [0,\infty]$ is lower semi-continuous and then show that $c \colon \CM_{\bar{\dN}}(Z) \times \CP(\CM_{\bar{\dN}}(Z)) \to [0,\infty]$ is lower semi-continuous.
  Let $\{ \xi_{n};\; n \in \dN \} \cup \{ \xi \} \subset \CM_{\dN}(Z)$ and $\{ \Pi_{n};\; n \in \dN \} \cup \{ \Pi \} \subset \CP(\CM_{\dN}(Z))$ such that, as $n \to \infty$, $(\xi_{n})$ converges weakly to $\xi$, and $(\Pi_{n})$ converges weakly (with respect to the weak topology on $\CM_{\dN}(Z)$) to $\Pi$.
Let us show that $\liminf_{n\to +\infty} c(\xi_n,\Pi_n) \geq c(\xi,\Pi)$.
By definition there exists $q \in \dN$ and $x_{1}, \dots, x_{q} \in Z$ such that $\xi = \sum_{i=1}^{q} \delta_{x_{i}}$.
According to the proof of \cref{lem:configuration}, we can without loss of generality assume that $\xi_{n} = \sum_{i=1}^{q} \delta_{x_{i}^{n}}$ where $x_{i}^{n} \to x_{i}$ when $n \to \infty$. Let $p$ be the number of points in the support of $\xi$ and write $\mathrm{Supp} (\xi) =\{a_1,\ldots,a_p\}$ (with therefore pairwise distinct $a_i's$) and denote by $k_j = \xi(\{a_j\})$, for all $1\leq j\leq p$. For all $j\in \{1,\ldots,p\}$ denote by $I_j = \{i : x_i = a_j\}$ and by $S_j^n = \{x_i^n : i \in I_j\}$.
By Jensen's inequality, and the fact that $\alpha$ and $[0,1] \ni u \mapsto {[1-u]}_{+}$ are convex, we find that
\begin{equation*}
  \xi(a_{j}) \sum_{z \in S_{n}^{j}} \alpha\left(\int {\left[1- \frac{\chi(z)}{\xi_{n}(z)}\right]}_{+} \Pi_{n}(\dd \chi) \right) \frac{\xi_{n}(z)}{\xi(a_{j})} \geq \xi(a_{j}) \alpha\left(\int {\left[1 - \frac{\sum_{z \in S_{n}^{j}} \chi(z)}{\xi(a_{j})} \right]}_{+} \Pi_{n}(\dd \chi)\right).
\end{equation*}
Fix some $\varepsilon>0$; for $n$ large enough, one has $S_j^n \subset \bar{B}(a_j,\varepsilon)$ for all $1\leq j\leq p$, where $\bar{B}(a,\varepsilon)$ denotes the closed ball of center $a$ and radius $\varepsilon$.
Since the points of $S_{n}^{j}$ are all distinct, we find $\sum_{z \in S_{n}^{j}} \chi(z) = \chi(S_{n}^{j}) \leq \chi(\bar{B}(a_{j}, \varepsilon))$.
Hence, by the monotonicity of $\alpha$:
\begin{equation*}
  \xi(a_{j}) \alpha\left(\int {\left[1 - \frac{\sum_{z \in S_j^n }\chi(z)}{k_j}\right]}_+\Pi_n(\dd\chi)\right) 
  \geq \xi(a_{j}) \alpha\left(\int {\left[1 - \frac{\chi(\bar{B}(a_j,\varepsilon))}{k_j}\right]}_+\Pi_n(\dd\chi)\right).
\end{equation*}
Since $\xi_{n} = \sum_{j =1}^{p} \sum_{z \in S_{n}^{j}} \xi_{n}(z) \delta_{z}$, summing the two previous estimates over $j \in [p]$ gives
\begin{equation*}
c(\xi_n,\Pi_n) \geq \int \alpha\left(\int {\left[1 - \frac{\chi(\bar{B}(a,\varepsilon))}{\xi(a)}\right]}_+\Pi_n(\dd\chi)\right)\dd\xi(a).
\end{equation*}
According to Portmanteau theorem, since $\bar{B}(a,\varepsilon)$ is a closed set, the map $\CM_{\dN}(Z) \ni \chi \mapsto \chi(\bar{B}(a,\varepsilon))$ is upper-semicontinuous (for the weak topology on $\CM_{\dN}(Z)$); hence, the lower semi-continuity (for the weak topology on $\CM_{\dN}(Z)$) of the map 
\begin{equation*}
  \CM_{\dN}(Z) \ni \chi \mapsto \left[1 - \frac{\chi(\bar{B}(a,\varepsilon))}{\xi(a)}\right]_+.
\end{equation*}
Therefore, using again the Portmanteau theorem, we obtain, for any fixed $a$, the lower semi-continuity of the map
\begin{equation}
  \Pi \mapsto \int {\left[1 - \frac{\chi(\bar{B}(a,\varepsilon))}{\xi(a)}\right]}_{+} \Pi(\dd\chi).
\end{equation}
Therefore, letting $n\to \infty$, and using that $\liminf a_{n} + b_{n} \geq \liminf a_{n} + \liminf b_{n}$ yield
\begin{equation*}
\liminf_{n\to +\infty}c(\xi_n,\Pi_n) \geq \int \alpha\left(\int {\left[1 - \frac{\chi(\bar{B}(a,\varepsilon))}{\xi(a)}\right]}_{+} \Pi(\dd\chi)\right)\dd\xi(a).
\end{equation*}
Finally, letting $\varepsilon \to 0$ and using the dominated convergence theorem, we have that
\begin{equation*}
  \chi(\bar{B}(a,\varepsilon)) \tto[\varepsilon][0]{} \chi(a).
\end{equation*}
Using again the dominated convergence theorem, we thus find that
\begin{equation*}
  \int {\left[1 - \frac{\chi(\bar{B}(a,\varepsilon))}{\xi(a)}\right]}_{+} \Pi(\dd\chi) \tto[\varepsilon][0]{} \int {\left[1 - \frac{\chi(a)}{\xi(a)}\right]}_{+} \Pi(\dd\chi).
\end{equation*}
Since $\alpha$ is lower semi-continuous, we finally find
\begin{equation*}
\liminf_{n\to +\infty}c(\xi_n,\Pi_n)\geq c(\xi,\Pi).
\end{equation*}
This proves that $c \colon \CM_{\dN}(Z) \times \CM_{\dN}(Z) \to [0,\infty]$ is weakly lower semi-continuous.
Now let us consider
\begin{equation*}
  \xi_{n} \tto{\CM_{\bar{\dN}}(Z)} \xi;\qquad \text{and} \qquad \Pi_{n} \tto{\CP(\CM_{\bar{\dN}}(Z))} \Pi.
\end{equation*}
Fix $o \in Z$, and, for $L > 0$, write $B_{L} = B(o,L)$.
By definition of the vague convergence, for all $L > 0$,
\begin{equation*}
  {\xi_{n}}_{\restriction B_{L}} \tto{weakly} \xi_{\restriction B_{L}}.
\end{equation*}
By \cite[Lemma 4.1]{Kallenberg}, $\CM_{\bar{\dN}}(Z) \ni \chi \mapsto \chi(\bar{B}(a, \varepsilon))$ is upper semi-continuous (for the vague topology on $\CM_{\bar{\dN}}(Z)$); hence, using again Portmanteau theorem, one gets that for any fixed $a$, the map $\CP(\CM_{\bar{\dN}}(Z)) \ni \Pi \mapsto \int {\left[1 - \frac{\chi(\bar{B}(a,\varepsilon))}{\xi(a)}\right]}_{+} \Pi(\dd\chi)$ is lower-semicontinuous.
Thus, arguing as before, we find that:
\begin{equation*}
  \liminf_{n \to \infty} c\left({\xi_{n}}_{\restriction B_{L}}, \Pi_{n}\right) \geq \int \alpha\left(\int {\left[1-\frac{\chi(a)}{\xi(a)}\right]}_{+} \Pi(\dd \chi)\right) 1_{B_{L}}(a) \xi(\dd a).
\end{equation*}
On the one hand, since $\xi_{n} \geq {\xi_{n}}_{\restriction B_{L}}$, we have that $c(\xi_{n}, \Pi_{n}) \geq c({\xi_{n}}_{\restriction B_{L}}, \Pi_{n})$.
On the other hand, by the monotone convergence theorem, we find that
\begin{equation*}
  \int \alpha\left(\int {\left[1-\frac{\chi(a)}{\xi(a)}\right]}_{+} \Pi(\dd \chi)\right) 1_{B_{L}}(a) \xi(\dd a) \tto[L]{} c(\xi, \Pi).
\end{equation*}
Combining these observations, we obtain that
\begin{equation*}
  \liminf_{n \to \infty} c(\xi_{n}, \Pi_{n}) \geq c(\xi, \Pi).
\end{equation*}
This concludes the proof.

\end{proof}

\section{Applications to concentration estimates}\label{section:concentration_binomial}
With the help of the transport-entropy inequalities proved in \cref{section:transport_ineq_point_process}, we derive concentration results for Poisson point processes and other mixed binomial processes.
Below, we highlight some particular cases.

\subsection{Generic concentration of measure inequalities}\label{sectopn:concentration_generic}

First let us draw some consequences of \cref{theorem:Talagrand_binomial}.
\begin{corollary}
  Let $\kappa \in \CP(\dN)$, $\mu \in \CP(Z)$, and $\eta \sim B_{\mu,\kappa}$. 
  Assume $\mu \in \CP(Z)$ satisfies \hyperref[equation:transport_entropy-omega]{\textnormal{($\lT_{\rho}$)}}, for some lower semi-continuous cost function $\rho \colon Z \times Z \to [0,\infty]$, with constant $a > 0$. 
  For all Borel sets $A_1,A_2 \subset \CM_{\dN}(Z)$ such that $\dP(\eta \in A_i)>0$ and $\dT_\rho(\Pi_1,\Pi_2)<+\infty$, where $\Pi_i = \frac{\1_{A_i}}{B_{\mu,\kappa}( A_i)}\,B_{\mu,\kappa}$, it holds that 
  \begin{equation*}
  \dP(\eta \in A_1) \dP(\eta \in A_2) \leq \exp\left(- \frac{1}{a}\inf_{\xi_1 \in A_1,\xi_2 \in A_2} \cT_\rho(\xi_1,\xi_2)\right).
  \end{equation*}
\end{corollary}
\begin{proof}
Simply apply the transport-entropy of \cref{theorem:Talagrand_binomial} to the probability measures $\Pi_1$ and $\Pi_2$ defined above and use that $\cH(\Pi_i | B_{\mu,\kappa}) = -\log \dP(\eta \in A_i)$, $i=1,2$, and the lower bound $\dT_\rho(\Pi_1,\Pi_2) \geq \inf_{\xi_1 \in A_1,\xi_2 \in A_2} \cT_\rho(\xi_1,\xi_2)$.
Details of this argument can be found in \textit{e.g} \cite{GozlanLeonard}.
\end{proof}
Due to the requirement $\dT_\rho(\Pi_1,\Pi_2)<+\infty$, this concentration bound seems delicate to manipulate. 

Now let us examine the consequences of \cref{theorem:Marton_binomial} and \cref{theorem:transport_entropy_poisson}.
To state them, let us recall some definitions from \cite{GRSTGeneral} adapted to our Poisson setting.
Let $c \colon \CM_{\bar{\dN}}(Z) \times \CP(\CM_{\bar{\dN}}(Z)) \to [0,\infty]$ be a cost.
For a Borel set $A \subset \CM_{\bar{\dN}}(Z)$, we write
\begin{equation}\label{equation:convex_distance}
  c_{A}(\xi) = \inf_{\Pi(A) = 1} c(\xi,\Pi), \qquad \forall \xi \in \CM_{\bar{\dN}}(Z),
\end{equation}
for the \emph{Talagrand convex distance} associated with $A$ and $c$ (which, despite its name, is not really a distance), and
\begin{equation}\label{equation:convex_enlargement}
  A_{r} = \{ x \in \CM_{\bar{\dN}}(Z) : c_{A}(x) \leq r \},\qquad r\geq0,
\end{equation}
for the \emph{enlargement} of $A$ with respect to the Talagrand convex distance.

With these notations, we state the following concentration result.
\begin{corollary}\label{theorem:concentration_binomial}
  Let $t \in (0,1)$ and $\alpha_{t}$ be the function defined in \cref{eq:Dembo}.
\begin{enumerate}[(a)]
\item Let $\eta$ be a binomial process (of size $n$).
  Then for every Borel set $A \subset \CM_{n}(Z)$
\begin{equation*}
  {\dP(\eta \in A)}^{\frac{1}{t}} {\dP(\eta \not \in A_{r})}^{\frac{1}{1-t}} \leq \e^{-r},\qquad \forall r\geq0,
\end{equation*}
where $A_{r}$ is the enlargement of $A$ (in $\CM_{n}(Z)$) given in \cref{equation:convex_enlargement} for the choice of $c = c_{t}$ given in \cref{equation:cost_binomial}, that is
\begin{equation*}
  c(\xi, \Pi) = \int \alpha_{t}\left(\int {\left[1 - \frac{\chi(x)}{\xi(x)}\right]}_{+} \Pi(\dd \chi)\right) \xi(\dd x), \quad \xi \in \CM_{n}(Z),\, \Pi \in \CP(\CM_{n}(Z)).
\end{equation*}
\item Let $\eta$ be a Poisson point process with finite intensity measure $\nu \in \CM_{b}(Z)$.
  Then for every Borel set  $A \subset \CM_{\dN}(Z)$
\begin{equation*}
  {\dP(\eta \in A)}^{\frac{1}{t}} {\dP(\eta \not \in A_{r})}^{\frac{1}{1-t}} \leq \e^{-r},\qquad \forall r\geq0,
\end{equation*}
where $A_{r}$ is the enlargement of $A$ (in $\CM_{\dN}(Z)$) given in \cref{equation:convex_enlargement} for the choice of $c = c_{t}$ given in \cref{equation:cost_poisson}, that is
\begin{equation*}
  c(\xi, \Pi) = \int \alpha_{t}\left(\int {\left[1 - \frac{\chi(x)}{\xi(x)}\right]}_{+} \Pi(\dd \chi)\right) \xi(\dd x) \quad \xi \in \CM_{\dN}(Z),\, \Pi \in \CP(\CM_{\dN}(Z)).
\end{equation*}
\item Let $\eta$ be a Poisson point process with $\sigma$-finite intensity measure $\nu$.
  Then for every Borel set  $A \subset \CM_{\bar{\dN}}(Z)$
\begin{equation*}
  {\dP(\eta \in A)}^{\frac{1}{t}} {\dP(\eta \not \in A_{r})}^{\frac{1}{1-t}} \leq \e^{-r},\qquad \forall r\geq0,
\end{equation*}
where $A_{r}$ is the enlargement of $A$ (in $\CM_{\bar{\dN}}(Z)$) given in \cref{equation:convex_enlargement} for the choice of $c = c_{t}$ given in \cref{equation:cost_poisson}, that is
\begin{equation*}
  c(\xi, \Pi) = \int \alpha_{t} \left(\int {\left[1 - \frac{\chi(x)}{\xi(x)}\right]}_{+} \Pi(\dd \chi)\right) \xi(\dd x) \quad \xi \in \CM_{\bar{\dN}}(Z),\, \Pi \in \CP(\CM_{\bar{\dN}}(Z)).
\end{equation*}
\end{enumerate}
\end{corollary}

\begin{proof}
  Combine \cite[Theorem 5.1]{GRSTGeneral} with \cref{theorem:transport_entropy_binomial} or \cref{theorem:transport_entropy_poisson}.
\end{proof}

For further discussions, let us first recall some definitions and state a theorem of \cite{ReitznerConcentration} to be compared with \cref{theorem:concentration_binomial}.
Given a point measure $\xi \in \CM_{\bar{\dN}}(Z)$, we write $\CL^{2}(\xi)$ for the space of functions $g \colon Z \to \dR$ such that $g(x) = 0$ if $\xi(x)=0$ and such that:
\begin{equation*}
  {|g|}^{2}_{\CL^{2}(\xi)} := \sum_{x \in \xi} \xi(x) {g(x)}^{2} < \infty,
\end{equation*}
where the notation $x \in \xi$ is shorthand for $x \in \supp \xi$.
When equipped with the norm ${|\cdot|}_{\CL^{2}(\xi)}$, the linear space $\CL^{2}(\xi)$ is a Hilbert space.
Given another $\chi \in \CM_{\bar{\dN}}(Z)$, we write $\xi \setminus \chi$ for the point measure given by $\sum_{x\in \xi} {(\xi(x) - \chi(x))}_{+} \delta_{x}$.
For $A \subset \CM_{\bar{\dN}}(Z)$ and $\xi \in \CM_{\bar{\dN}}(Z)$, following \cite{ReitznerConcentration} (\cite{ReitznerConcentration} only works on $\CM_{\dN}(Z)$), we define
\begin{equation}\label{equation:d_A}
  d_{A}(\xi) = \sup_{|g|_{\CL^{2}(\xi)} \leq 1} \inf_{\chi \in A} \int g \dd (\xi \setminus \chi),
\end{equation}
where the supremum runs over non-negative $g$ only, and we also define
\begin{equation*}
  A_{r}^{d} = \{ \xi \in \CM_{\bar{\dN}}(Z):\, d_{A}(\xi) \leq r \},\qquad r \geq 0.
\end{equation*}
Then, we have the following extension of \cite[Theorem 1.1]{ReitznerConcentration}.
\begin{theorem}\label{theorem:Reitzner}
  Let $\eta$ be a Poisson point process on $Z$ with a $\sigma$-finite intensity.
  Then, for every Borel set $A \subset \CM_{\bar{\dN}}(Z)$,
\begin{equation*}
  \dP(\eta \in A) \dP(\eta \not \in A_{r}^{d}) \leq e^{-\frac{r^{2}}{4}}.
\end{equation*}
\end{theorem}
The case of a Poisson point processes with finite intensity measure is proved established in \cite{ReitznerConcentration}, whereas the case of an intensity with infinite mass is new.
The proof of \cref{theorem:Reitzner} follows immediately from \cref{theorem:concentration_binomial} (with $t=1/2$) and the following lemma (using the fact that $\alpha_{1/2}(u) \geq u^2/2$). 

 \begin{lemma}
   With the notation introduced above, and the functions $\alpha(u) = \frac{u^{2}}{2}$, $u\geq0$, and
   \begin{equation*}
     c(\xi, \Pi) = \int \alpha\left( \int {\left[1 - \frac{\chi(x)}{\xi(x)}\right]}_{+} \Pi(\dd \chi)\right) \xi(\dd x), \qquad \xi \in \CM_{\bar{\dN}}(Z), \Pi \in \CP(\CM_{\bar{\dN}}(Z)).
   \end{equation*}
   For any $A \subset \CM_{\bar{\dN}}(Z)$, it holds $A_{2\sqrt{r}}^{d} = A_{r}$; equivalently $c_{A}(\xi) = \frac{1}{2} {d_{A}(\xi)}^{2}$, for all $\xi \in \CM_{\bar{\dN}}(Z)$.
\end{lemma}

\begin{proof}
  The argument is quite classical and goes back to Talagrand.
  We give a proof for completeness.
  First, we recall this well-known duality formula on Hilbert spaces: if $\CH$ is an Hilbert space with inner product $\langle \cdot, \cdot \rangle$ and induced norm $| \cdot |$, we have 
  \begin{equation*}
  |x| = \sup_{|y| \leq 1} \langle x, y \rangle.
  \end{equation*}
  Second, we recall this well-known fact about randomization of infimum
\begin{equation*}
    \inf_{x \in A} f(x) = \inf_{X \in A} \dE [f(X)],
\end{equation*}
  where the second infimum runs on random variables $X$ taking values in $A$.\\
  Third, for two convex spaces $C_{1}$ and $C_{2}$, consider a bilinear functional $\Lambda \colon C_{1} \times C_{2} \to \dR$ that is upper semi-continuous and concave in its first variable, and lower semi-continuous and convex in its second variable.
  If $C_{1}$ is compact, then by the min-max theorem \cite[Corollary 3.3]{SionMinMax}:
\begin{equation*}
    \sup_{C_{1}} \inf_{C_{2}} \Lambda = \inf_{C_{2}} \sup_{C_{1}} \Lambda.
\end{equation*}
We equip $\CL^{2}(\xi)$ with the weak topology $\sigma(\CL^{2}(\xi), \CL^{2}(\xi))$ defined by duality, and we write $\CL^{2}_{+}(\xi)$ for the convex cone of non-negative $g \in \CL^{2}(\xi)$.
We write $\CP_{\xi}$ for the convex set of all $\Pi \in \CP(\CM_{\bar{\dN}}(Z))$ such that
\begin{equation*}
  \left(Z \ni x \mapsto \int {\left[1 - \frac{\chi(x)}{\xi(x)}\right]}_{+} \Pi(\dd \chi) \right) \in \CL^{2}(\xi).
\end{equation*}
We have that
\begin{equation*}
  \CP_{\xi} = \{ \Pi \in \CP(\CM_{\bar{\dN}}(Z)) : c(\xi, \Pi) < \infty \}.
\end{equation*}
We now define the function $\Lambda \colon \CL^{2}_{+}(\xi) \times \CP(\CM_{\bar{\dN}}(Z)) \to [0,\infty]$ by
\begin{equation*}
  \Lambda(g, \Pi) = \int g \dd (\xi \setminus \chi) \Pi(\dd \chi) = \int g(x) \int {\left[ 1 - \frac{\chi(x)}{\xi(x)}\right]}_{+} \Pi(\dd \chi) \xi(\dd x), \qquad \text{if}\ \Pi \in \CP_{\xi},
\end{equation*}
and $\Lambda(g, \Pi) = \infty$ otherwise.
We claim that $\Lambda$ is upper semi-continuous and concave in its first variable.
The concavity is immediate.
Let us show the upper semi-continuity.
Let $\Pi \in \CP(\CM_{\bar{\dN}}(Z))$ and $\{g_{n};\; n \in \dN\} \cup \{g\} \subset \CL^{2}_{+}(\xi)$ such that:
\begin{equation*}
  g_{n} \tto{\sigma(\CL^{2}(\xi), \CL^{2}(\xi))} g,
\end{equation*}
Without loss of generality, we can assume that $\Pi \in \CP_{\xi}$ (otherwise $\Lambda(g, \Pi) = \infty$ and there is nothing to prove).
From the definition of the $\sigma(\CL^{2}(\xi), \CL^{2}(\xi))$ convergence, and the fact that $\Pi \in \CP_{\xi}$, we find
\begin{equation*}
  \limsup \Lambda(g_{n}, \Pi) = \lim \Lambda(g_{n}, \Pi) = \Lambda(g, \Pi).
\end{equation*}
We now claim that $\Lambda$ is lower semi-continuous and convex in its second variable.
The convexity is also immediate.
By \cite[Theorem 4.1]{Kallenberg}, for all $x \in \xi$ and $g\in \CL^{2}_{+}(\xi)$, the function
\begin{equation*}
  \CM_{\bar{\dN}}(Z) \ni \chi \mapsto g(x) {(\xi(x) - \chi(x))}_{+},
\end{equation*}
is lower semi-continuous; hence by Portmanteau's theorem:
\begin{equation*}
  \CP(\CM_{\bar{\dN}}(Z)) \ni \Pi \mapsto \int g(x) {(\xi(x) - \chi(x))}_{+} \Pi(\dd \chi),
\end{equation*}
is lower semi-continuous.
By Tonelli's theorem, we find that
\begin{equation*}
  \Lambda(g, \Pi) = \sum_{x \in \xi} \int g(x) {(\xi(x) - \chi(x))}_{+} \Pi(\dd \chi).
\end{equation*}
We conclude since a sum of lower semi-continuous functions is again lower semi-continuous.
By the Banach-Alaoglu theorem, $\CB$, the unit ball of $\CL^{2}(\xi)$, is weakly-$*$ compact; hence, since $\CL^{2}(\xi)$ is reflexive, $\CB$ is weakly compact.
Let us also write $\CB_{+}$ for $\CB \cap \CL^{2}_{+}(\xi)$.
We can therefore apply the randomization of infimum, the min-max theorem (remark that $\{\Pi \in \CP(\CM_{\bar{\dN}}(Z)) : \Pi(A) = 1\}$ is convex), and the duality formula for the norm in Hilbert spaces (we recall these three arguments at the beginning of the proof) to obtain that
\begin{align*}
  d_{A}(\xi) &=\sup_{g \in \CB_+} \inf_{\chi \in A}   \int g \dd (\xi \setminus \chi)\\
             & =  \sup_{g \in \CB_{+}}\inf_{\Pi(A) = 1} \int g(x) \left(\int {\left[1- \frac{\chi(x)}{\xi(x)}\right]}_{+} \Pi(\dd \chi) \right) \xi(\dd x)\\
             &= \inf_{\Pi(A) = 1} \sup_{g \in \CB_{+}} \int g(x) \left(\int {\left[1 - \frac{\chi(x)}{\xi(x)}\right]}_{+} \Pi(\dd \chi) \right) \xi(\dd x)\\
             &= \inf_{\Pi(A) = 1} {\left(\int {\left(\int {\left[1 - \frac{\nu(x)}{\xi(x)}\right]}_{+} \Pi(\dd \nu) \right)}^{2} \xi(\dd x)\right)}^{1/2}\\
             &= {(2 c_{A}(\xi))}^{1/2}.
\end{align*}
\end{proof}

\subsection{Concentration of measure for convex functionals}\label{section:concentration_convex}
On $\dR^{p}$, Marton's inequality implies a deviation inequality for convex $1$-Lipschitz functions (see \cref{eq:dev-ineq2}).
The goal of this section is twofold:
\begin{enumerate}[(i)]
  \item Introduce a reasonable notion for convex and Lipschitz functions on $\CM_{\bar{\dN}}(Z)$ and show that those functions also satisfy a deviation estimate.
  \item Show how this new deviation estimate sheds light on concentration results obtained for $U$-statistics by \cite{BachmannReitzner}.
\end{enumerate}
Despite the lack of a geodesic structure and of a metric structure on $\CM_{\bar{\dN}}(Z)$, we define convex and Lipschitz functionals in a rather straightforward way thanks to an analogy with the euclidean case.
Recall that for a differentiable function $f \colon \dR^{p} \to \dR$, $f$ is convex if and only if its Hessian is non-negative, while $f$ is Lipschitz if and only if the norm of the gradient of $f$ is bounded by a constant.

On the Poisson space, the \emph{difference operators} provide an ersatz of a differentiable structure.
For a functional $F$ defined on $\CM_{\bar{\dN}}(Z)$, we define these operators as follows:
\begin{align*}
  & D_{x}^{-}F(\xi) = F(\xi) - F(\xi - \delta_{x}), \qquad \forall x \in \xi; \\
  & D_{x}^{+}F(\xi) = F(\xi + \delta_{x}) - F(\xi), \qquad \forall x \in Z.
\end{align*}
In the sequel, we show concentration of measure for \emph{convex} functionals $F \colon \CM_{\bar{\dN}}(Z) \to \dR$.
By definition, we say that a functional $F$ is \emph{non-decreasing} if
\begin{equation}\label{eq:monotonic}
  \forall \xi, \chi \in \CM_{\bar{\dN}}(Z),\, (\xi \leq \chi \implies F(\xi) \leq F(\chi)).
\end{equation}
We can alternatively state this monotonicity property with the $D^{-}$ or $D^{+}$ operators: $F$ is non-decreasing if and only if $D_{x}^{-}F(\xi) \geq 0$ for all $\xi \in \CM_{\bar{\dN}}(Z)$ and all $x \in \xi$ if and only if $D_{x}^{+}F(\xi) \geq 0$ for all $\xi \in \CM_{\bar{\dN}}(Z)$ and all $x \in Z$. 
We say that the functional $F$ is \emph{convex} whenever $D_{x}^{+}F$ is non-decreasing for all fixed $x \in Z$.
Alternatively, $F$ is convex if and only if its \enquote{second derivative} satisfies: $D_{xy}^{+}F(\xi) \geq 0$ for all $\xi \in \CM_{\bar{\dN}}(Z)$, and all $x$ and $y \in Z$.
The class of non-decreasing convex functionals contains in particular the class of $U$-statistics with non-negative kernel.
The \emph{$U$-statistics} of order $q \in \dN_{>0}$ associated with a symmetric kernel $h \colon Z^{q} \to \dR_{+}$ is a functional of the form
\begin{equation*}
  F(\xi) = \sum_{i_{1}, \dots, i_{q} \leq n}^{\ne} h(x_{i_{1}}, \dots, x_{i_{q}}) := \xi^{(q)}(h), \qquad \forall \xi = \sum_{i=1}^{n} \delta_{x_{i}} \in \CM_{\bar{\dN}}(Z),
\end{equation*}
where the superscript $\ne$ indicates summation over pairwise disjoint indices.
By Mecke's formula \cite[Theorem 4.4]{LastPenrose}, when $h \in \CL^{1}(\nu^{q})$, we have that $\dE |F(\eta)| \leq \nu^{q}(|h|) < \infty$, and the random variable $F(\eta)$ is almost surely finite.
It is immediate that, for all $x \in Z$, $D_{x}^{+}F(\xi) = q \xi^{(q-1)}(h(x,\cdot))$.
Thus, we see that when $h \geq 0$, then the $U$-statistics $F$ is non-decreasing in the sense of \cref{eq:monotonic}.
We also have that $D_{y}^{+}D_{x}^{+}F(\xi) = q(q-1) \xi^{(q-2)}(h(x,y,\cdot))$.
So that, when $h \geq 0$, the $U$-statistics $F$ is also convex in our sense.
Similarly, we define a notion of being Lipschitz by imposing a bound on the square of the \enquote{norm of the gradient} which in our setting is given by $\sum_{\eta} {(D_{x}^{-}F(\eta))}^{2}$.

We now give a slight extension of a result of \cite{BachmannReitzner} to convex functionals and to Poisson point processes with intensity measure possibly having atoms.
\begin{theorem}[{\cite[Proposition 5.6]{BachmannReitzner}}]\label{theorem:BR} Suppose that $\eta$ is a Poisson process with a $\sigma$ finite intensity measure. 
Let $F \colon \CM_{\bar{\dN}}(Z) \to \dR_{+}$ be a non-negative convex functional such that there exist $\delta > 0$ and $\beta \in [0,2)$ such that
  \begin{equation}\label{equation:condition_sum}
    \sum_{x \in \eta} {(D_{x}^{-}F(\eta))}^{2} \leq \delta {F(\eta)}^{\beta} \qquad \text{a.s.}
  \end{equation}
  Let $m$ be a median of $F$.
  Then, for all $r \geq 0$, we have that
  \begin{align*}
    & \dP(F(\eta) \geq m + r) \leq 2 \exp\left( - \frac{r^{2}}{4 \delta {(r + m)}^{\alpha}}\right); \\
    & \dP(F(\eta) < m - r) \leq 2 \exp\left(- \frac{r^{2}}{4 \delta m^{\alpha}}\right).
  \end{align*}
\end{theorem}
The proof of the theorem goes simply by observing that \cite[Lemma 5.7]{BachmannReitzner} still holds for convex functionals.
The rest of the proof is similar and is omitted.
\begin{lemma}\label{lemma:difference_convex}
  Let $F \colon \CM_{\bar{\dN}}(Z) \to \dR_{+}$ be a convex functional.
  Then, for all $\xi$ and $\chi \in \CM_{\bar{\dN}}(Z)$, such that $\chi \leq \xi$:
  \begin{equation*}
    F(\xi) - F(\chi) \leq \int D_{x}^{-}F(\xi)\,(\xi \setminus \chi)(dx).
  \end{equation*}
\end{lemma}
\begin{proof}
  Without loss of generality, we assume that $\xi \ne 0$.
  Since $\chi \leq \xi$, we can write $\xi = \chi + \sum_{i=1}^{n} \delta_{x_i}$, where $n \in \dN \cup\{+\infty\}$ and the $x_i's$ are elements of $Z$. Then, denoting by $\chi_0 = \chi$ and for all $1\leq i \leq n-1$, $\chi_{i+1} = \chi_i + \delta_{x_{i+1}}$, it holds 
\begin{align*}
F(\xi) - F(\chi) & = \sum_{i=0}^{n-1} F(\chi_{i+1}) - F(\chi_i) \\
                 &= \sum_{i=0}^{n-1} D^+_{x_{i+1}}F(\chi_i) \leq \sum_{i=0}^{n-1} D^+_{x_{i+1}}F(\xi - \delta_{x_{i+1}}) = \sum_{i=0}^{n-1} D^-_{x_{i+1}}F(\xi),
\end{align*}
where the inequality comes from the fact that $D_x^+F$ is non-decreasing.
\end{proof}

In \cite{BachmannReitzner}, provided that $F$ satisfies an inequality as of the one of \cref{lemma:difference_convex}, they give a lower bound on $d_{A}(\eta)$ by some quantity depending on $F$ only.
Then, they use the deviation inequality of \cite{ReitznerConcentration} (see \cref{theorem:Reitzner}) involving the convex distance $d_{A}$ defined in \cref{equation:d_A} to conclude in the case of a finite intensity measure and use an approximation by thinning to conclude.
Following, our previous results, we could follow an alternative approach rather based on the distance $c_{A}$ defined in \cref{equation:convex_distance}.
The distance $c_{A}$ intimately relates to the infimum-convolution operator $R_{c}$ (see \cite[Section 5]{GRSTGeneral} for a discussion of this link in the general case).
We define $R_{c}$ in the next section to study logarithmic Sobolev inequalities linked to our transport-entropy inequality.
As shown by \cref{lemma:lower_bound_Rc}, we are only able to bound $R_{c}$ for convex functions so that with the approach based on $c_{A}$, we could not improve upon \cref{theorem:BR}.

\section{Modified logarithmic Sobolev inequality}\label{section:log_sob}
In this section we investigate the links between the transport inequality obtained in \cref{theorem:transport_entropy_poisson} and modified logarithmic Sobolev inequalities on the Poisson space. 

\subsection{Introduction}On the Euclidean space (and more generally on a Riemannian manifold), a celebrated result by Otto \& Villani \cite{OttoVillani} (see also \cite{BobkovGentilLedoux}) shows that a probability measure $\gamma \in \mathcal{P}(\dR^d)$ satisfies the quadratic Talagrand's inequality 
\begin{equation}\label{eq:Tal}
\cW_2^2(\nu,\gamma) \leq C \cH(\nu|\gamma),\qquad \forall \nu
\end{equation}
whenever $\gamma$ satisfies the logarithmic Sobolev inequality: for all function $f$ sufficiently smooth,
\begin{equation*}
  \Ent\left( \e^{f(X)} \right)\leq \frac{C}{4}\dE\left[ \e^{f(X)} {|\nabla f(X)|}^{2}\right], \qquad X \sim \gamma,
\end{equation*}
denoting 
\begin{equation*}
\Ent\left( \e^{Z}\right) = \dE[Z\e^Z] - \dE[Z]\log \dE[\e^Z],
\end{equation*}
for any random variable $Z$ such that $\dE[|Z| e^Z] <+\infty.$

The work by \cite{GRSTalagrandLogSob} supplements the Otto-Villani theorem: it shows that $\gamma$ satisfies \cref{eq:Tal} if and only if $\gamma$ satisfies the following \emph{restricted} logarithmic Sobolev inequality
\begin{equation*}
  \Ent(\e^{f(X)}) \leq C' \dE \left[ |\nabla f(X)|^{2} \e^{f(X)}\right], \qquad X \sim \gamma,
\end{equation*}
for all $f \colon \dR^{d} \to \dR$ such that $x \mapsto f(x) + C'' {|x|}^{2}$ is convex, where $C',C'' >0$ depend quantitatively on $C$.
The reference \cite{GRSTalagrandMetric} further investigates such an equivalence between transport inequalities and modified versions of the logarithmic Sobolev inequality, and proves that on a complete separable metric space $(E,d)$, a probability measure $\gamma \in \CP(E)$ satisfies \cref{eq:Tal} (where $\cW_2$ is defined with respect to the metric $d$) if and only if there exists $C',\lambda > 0$ (that can be precisely related to $C$) such that, for all bounded measurable function $f$
\begin{equation*}
  \Ent \left(\e^{f(X)} \right) \leq C'\dE \left[\e^{f(X)} (f(X) - Q^{\lambda}f(X)) \right], \qquad X \sim \gamma,
\end{equation*}
where $Q^\lambda$ is the infimum convolution operator defined by
\begin{equation*}
  Q^{\lambda}f(x) = \inf_{y \in E} \left\{f(y) + \lambda {d(x,y)}^{2}\right\},\qquad x \in E.
\end{equation*}
We refer to \cite{GRSTGeneral, ShuHJGraph, ShuS18} for other results connecting transport inequalities of the form \cref{equation:transport_entropy-c} and variants of the logarithmic Sobolev inequality.

Our goal in what follows is to do a first step in extending theses results to the Poisson framework. 
\subsection{Transport inequalities and variants of the log-Sobolev inequality on the Poisson space}

In \cite{GRSTGeneral}, it is shown that a transport-entropy inequality of the form \cref{equation:transport_entropy-c}, involving some cost function $c \colon E \times \CP(E)$ with $E$ being an arbitrary Polish space, always implies some sort of logarithmic Sobolev inequality, whose energy term contains a ``gradient'' defined using the following infimum-convolution operator $R_{c}$
\begin{equation*}
  R_{c}F(\xi) = \inf_{\Pi \in \CP(E)} \left\{ \Pi(F) + c(\xi, \Pi) \right\}, \qquad \forall \xi \in E,
\end{equation*}
for all $F : E \to \dR$.

In our case, \cref{theorem:transport_entropy_poisson} with $t=1$ (and $E =  \CM_{\bar{\dN}}(Z)$) combined with \cite[Theorem 3.8]{GRSTGeneral} immediately gives the following result.
\begin{theorem}[Infimum-convolution logarithmic Sobolev inequality]\label{theorem:log_sob_Rc}
  Let $\nu$ be a $\sigma$-finite measure on $Z$ and let $\eta \sim \Pi_{\nu}$.
  Then, for all $F \colon \CM_{\bar{\dN}}(Z) \to \dR$ such that $\dE\left[ |F|(\eta) \e^{F(\eta)}\right] < \infty$, we have, for all $\lambda \in (0,1)$:
  \begin{equation}\label{equation:log_sob_Rc}
    \Ent \left(\e^{F(\eta)} \right) := \dE \left[ \e^{F(\eta)} F(\eta) \right] - \dE \left[ F(\eta)\right] \dE \left[ \e^{F(\eta)} \right] \leq \frac{1}{1 - \lambda} \dE \left[(F(\eta) - R_{\lambda c_{1}}F(\eta)) \e^{F(\eta)}\right]
  \end{equation}
  where $c_{1}$ is the cost given in \cref{equation:cost_poisson} associated with $\alpha_{1}$ given in \cref{eq:Dembo}.
\end{theorem}
We do not know if Inequality \cref{equation:log_sob_Rc} implies back the transport inequality of \cref{theorem:transport_entropy_poisson} (for $t=1$).

In the following result, we deduce from \cref{theorem:log_sob_Rc} a modified logarithmic Sobolev inequality reminiscent of Wu's inequality \cite{WuLSI} in restriction to the class of non-decreasing convex functions.
For all $\lambda \in (0,1)$, we will denote by
\begin{equation*}
  \phi_{\lambda}(s) = \frac{s}{1-\lambda} - \frac{\lambda}{1-\lambda} \log\left(1+ \frac{s}{\lambda}\right), \qquad s \geq 0.
\end{equation*}
We also consider
\begin{equation*}
  \phi_{0}(s) = \lim_{\lambda \to 0} \phi_{\lambda}(s) = s, \qquad s \geq 0.
\end{equation*}
\begin{corollary}\label{theorem:log_sob_monotonic}  Let $\nu$ be a $\sigma$-finite measure on $Z$ with no atoms and let $\eta \sim \Pi_{\nu}$.
  Let $F \colon \CM_{\bar{\dN}}(Z) \to \dR$ be convex non-decreasing.
  Then, for all $0 \leq \lambda < 1$:
  \begin{equation*}
    \Ent\left( \e^{F(\eta)}\right) \leq \dE \left[\e^{F(\eta)} \int \phi_{\lambda}(D_{x}^{-}F(\eta)) \eta(\dd x)\right].
  \end{equation*}
\end{corollary}

As mentioned above, this inequality is close to the modified logarithmic Sobolev inequality by Wu \cite{WuLSI}: for all $F \colon \CM_{\bar{\dN}}(Z) \to \dR$,
\begin{equation*}
  \Ent\left( \e^{F(\eta)} \right)\leq \dE\left[ \e^{F(\eta)} \int \phi_{w}(D_{x}^{-}F(\eta)) \eta(\dd x)\right],
\end{equation*}
where
\begin{equation*}
  \phi_{w}(s) = \e^{-s} + s - 1,\qquad s\geq0.
\end{equation*}
The functions $\phi_w(s)$ and $\phi_\lambda(s)$, $\lambda \in (0,1)$, are of the same order : quadratic for small values of $s$ and linear for large values $s$. Nevertheless, one can observe that for $s$ sufficiently large $\phi_{w}(s) < \phi_{\lambda}(s)$, so that our inequality does not improve upon the one by \cite{WuLSI}.

To prove \cref{theorem:log_sob_monotonic}, we will need the following lemma.

\begin{lemma}\label{lemma:lower_bound_Rc}Let $\nu$ be a $\sigma$-finite measure on $Z$ with no atoms and let $\eta \sim \Pi_{\nu}$.
  Let $\alpha \colon\dR \to [0,\infty]$ be any convex lower semi-continuous function.
  Consider the function $c \colon \CM_{\bar{\dN}}(Z) \times \CP(\CM_{\bar{\dN}}(Z)) \to [0,\infty]$ defined by
  \begin{equation*}
    c(\xi, \Pi) = \int \alpha \left(\int {\left[1-\frac{\chi(x)}{\xi(x)}\right]}_{+} \Pi(\dd \chi) \right) \xi(\dd x).
  \end{equation*}
  For any non-decreasing convex $F \colon \CM_{\bar{\dN}}(Z) \to \dR_{+}$, we have that
  \begin{equation}\label{eq:lowerbRc}
  F(\eta)- R_{c}F(\eta)  \leq \int \alpha^{*}(D_{x}^{-}F(\eta)) \,\eta(dx) \qquad \text{a.s.},
  \end{equation}
  where  $\alpha^{*}(s) = \sup_{u} \{ us - \alpha(u)\}$ is the Fenchel-Legendre conjugate of $\alpha$.
\end{lemma}
\begin{proof}
Since the intensity measure of $\eta$ has no atoms, with probability one $\eta$ is simple, i.e for all $x \in Z$, $\eta(x) \in \{0,1\}$. Let us consider $\xi = \sum_{i=1}^{n} \delta_{x_{i}} \in \CM(Z)$, with pairwise distinct $x_i's$ and $n = \xi(Z) \in \dN \cup \{\infty\}$; we have the following simplification:
\begin{equation*}
  c(\xi, \Pi) = \sum_{x \in \xi} \alpha(\dP(\chi(x) = 0)).
\end{equation*}
Now we observe that
\begin{equation*}
  \sum_{x \in \xi} \alpha(\dP(\chi(x)=0)) = \sum_{x \in \xi} \alpha(\dP((\chi \cap \xi)(x) = 0)),
\end{equation*}
where $\chi \cap \xi = \sum_{x \in \xi} \min(\xi(x), \chi(x)) \delta_{x}$.
Since $\chi \geq \chi \cap \xi$ and $F$ is non-decreasing, we have that $F(\chi) \geq F(\chi \cap \xi)$.
Therefore, denoting by $\Xi(\xi)$ the set of random variables $\chi$ taking values in $\CM_{\bar{\dN}}(Z)$ such that $\dP(\chi \leq \xi) = 1$, we get
\begin{equation*}
  R_{c}F(\xi) = \inf_{\chi \in \Xi(\xi)}\left\{ \dE [ F(\chi)] + \int \alpha(\dP(\chi(x) = 0)) \xi(\dd x) \right\}.
\end{equation*}
Invoking \cref{lemma:difference_convex}, we thus find that
\begin{equation*}
  R_{c}F(\xi) - F(\xi) \geq \inf_{\chi \in \Xi(\xi)} \left\{ \int \alpha(\dP(\chi(x)=0)) \xi(\dd x) - \dE \left[ \int D_{x}^{-}F(\xi) (\xi\setminus \chi)(\dd x) \right] \right\}.
\end{equation*}
Now observe that a random measure $\chi \in \Xi(\xi)$ if and only if it is of the form
\begin{equation*}
  \chi = \sum_{i=1}^{n} (1-\varepsilon_{i}) \delta_{x_{i}},
\end{equation*}
for some Bernoulli random variables $(\varepsilon_{i})_{1\leq i\leq n}$.
Therefore
\begin{align*}
R_{c}F(\xi) - F(\xi) & \geq \inf_{\varepsilon = (\varepsilon_{i})_{1\leq i\leq n}} \sum_{i=1}^{n} \left(\alpha(\dE \varepsilon_{i}) - D_{x_{i}}^{-}F(\xi) \dE \varepsilon_{i}\right) \\
                         &= \inf_{p \in {[0,1]}^{n}} \sum_{i=1}^{n} (\alpha(p_{i}) - D_{x_{i}}^{-}F(\xi)p_{i})\\
			& \geq \sum_{i=1}^{n} \inf_{u \in \dR} (\alpha(u) - D_{x_{i}}^{-}F(\xi) u) \\
                         &= - \sum_{i=1}^{n} \alpha^{*}(D_{x_{i}}^{-}F(\xi)),
\end{align*}
which completes the proof.
\end{proof}

\begin{proof}[Proof of \cref{theorem:log_sob_Rc}]
Since, for $\lambda > 0$, ${(\lambda \alpha_1)}^{*}(s) = \lambda \alpha_1^{*} (t/\lambda)$, $t \in \dR$, \cref{lemma:lower_bound_Rc} gives that
  \begin{equation*}
    F(\eta) - R_{\lambda c_1}F(\eta) \leq \lambda \int \alpha_1^{*} \left(\frac{D_{x}^{-} F(\eta)}{\lambda}\right) \eta(\dd x),
  \end{equation*}
  for all non-decreasing convex function $F$.
  Therefore, by \cref{theorem:log_sob_Rc}, for such $F$ it holds
  \begin{equation*}
    \Ent \left( \e^{F(\eta)} \right) \leq \frac{\lambda}{1-\lambda} \dE\left[  \e^{F(\eta)} \int \alpha_1^{*} \left(\frac{D_{x}^{-}F(\eta)}{\lambda}\right) \eta(\dd x) \right], \qquad \forall \lambda \in (0,1).
  \end{equation*}
  A simple calculation shows that $\alpha_1^*(s)= s - \log(1+s)$ for all $s>-1$.  This concludes the proof.

\end{proof}

\section{Some open questions}
\subsection{From modified logarithmic Sobolev to the transport-entropy}
We ask whether it is possible to recover our \cref{theorem:Marton_binomial} directly from the modified logarithmic Sobolev inequalities of \cite{WuLSI} or from the infimum-convolution logarithmic Sobolev inequality \cref{equation:log_sob_Rc}.
Following the ideas of \cite{BobkovGentilLedoux}, doing so would require a better understanding of the infimum-convolution operator $R_{c}$ on the Poisson; as well as new techniques regarding Hamilton-Jacobi equations in the setting of generalized optimal transport, a question that has its own independent interest.

\subsection{Links with displacement convexity}
\cite{GRSTDisplacement} introduces a notion of discrete displacement convexity of the entropy for finite graphs with respect to $\widetilde{\cT}_{2}$.
This displacement convexity, that is one of the many possible adaptations of the Lott-Sturm-Villani synthetic curvature bound to the discrete setting, entails many other functional inequalities, such as a modified logarithmic Sobolev inequality and a transport-entropy inequality.
In view of \cref{thm:Marton}, defining, in the spirit of \cite{GRSTDisplacement}, a notion of discrete curvature on the Poisson space would probably involve our cost $\dM$, and we hope that the present work could help clarifying the situation.

\subsection{Interacting point processes}
The mixed binomial processes and Poisson point processes studied in this work exhibit a very strong independence.
By definition, if $A$ and $B$ are disjoint measurable sets with finite $\nu$-measure then $\eta(A)$ is independent of $\eta(B)$.
More generally, the $q$-points correlation measure of a mixed binomial process with sampling measure $\mu$ is proportional to $\mu^{q}$.
We use this independence by identifying a mixed binomial process with the image of independent and identically distributed random variables.
Developing tools to study transport-entropy inequalities for point processes not relying on the independence seems an attractive path, as it would allow to consider other point processes such as determinantal point processes.
Let us point out that, to the best of our knowledge, the understanding of transport inequalities for point processes with interaction remains for the moment very partial.
There exists some $W_{1}$ transport-entropy inequality for Gibbs interaction as shown by \cite{MaShenWangWu}.
In the setting of empirical measures (normalized point process), we can also mention the recent work by \cite{ChafaiHardyMaida} about Coulomb gases.

\printbibliography
\end{document}